\documentclass[reqno,12pt]{amsart}
\usepackage[active]{srcltx}
\usepackage {color}
\usepackage{amsmath,amsthm,amssymb}
\usepackage{epsfig}
\usepackage{graphicx}
\usepackage{amssymb}
\usepackage{latexsym}
\usepackage{pdfpages}

\usepackage{enumitem}

\usepackage{ulem}

\usepackage{ifpdf}
\newcommand{\res}{\!\!\mathop{\hbox{
                                \vrule height 7pt width .5pt depth 0pt
                                \vrule height .5pt width 6pt depth 0pt}}
                                \nolimits}

\ifpdf %if using pdfLaTeX in PDF mode
  \usepackage[hidelinks]{hyperref}
\else %if using LaTeX or pdfLaTeX in DVI mode
\fi %%% fin del ifpdf

 \usepackage[usenames,dvipsnames]{pstricks}
 \usepackage{epsfig}
 \usepackage{pst-grad} % For gradients
 \usepackage{pst-plot} % For axes

\usepackage{color}

\newtheorem{theorem}{Theorem}[section]
\newtheorem{lemma}[theorem]{Lemma}
\newtheorem{definition}[theorem]{Definition}
\newtheorem{proposition}[theorem]{Proposition}

\newtheorem{example}[theorem]{Example}
\newtheorem*{theorem*}{\it Theorem}

\def\vint_#1{\mathchoice%
          {\mathop{\kern 0.2em\vrule width 0.6em height 0.69678ex depth -0.58065ex
                  \kern -0.8em \intop}\nolimits_{\kern -0.4em#1}}%
          {\mathop{\kern 0.1em\vrule width 0.5em height 0.69678ex depth -0.60387ex
                  \kern -0.6em \intop}\nolimits_{#1}}%
          {\mathop{\kern 0.1em\vrule width 0.5em height 0.69678ex
              depth -0.60387ex
                  \kern -0.6em \intop}\nolimits_{#1}}%
          {\mathop{\kern 0.1em\vrule width 0.5em height 0.69678ex depth -0.60387ex
                  \kern -0.6em \intop}\nolimits_{#1}}}
\def\vintslides_#1{\mathchoice%
          {\mathop{\kern 0.1em\vrule width 0.5em height 0.697ex depth -0.581ex
                  \kern -0.6em \intop}\nolimits_{\kern -0.4em#1}}%
          {\mathop{\kern 0.1em\vrule width 0.3em height 0.697ex depth -0.604ex
                  \kern -0.4em \intop}\nolimits_{#1}}%
          {\mathop{\kern 0.1em\vrule width 0.3em height 0.697ex depth -0.604ex
                  \kern -0.4em \intop}\nolimits_{#1}}%
          {\mathop{\kern 0.1em\vrule width 0.3em height 0.697ex depth -0.604ex
                  \kern -0.4em \intop}\nolimits_{#1}}}

\def\R{\mathbb R}

\numberwithin{equation}{section}

\def\D{\displaystyle}

\def\1{\raisebox{2pt}{\rm{$\chi$}}}

\def\Xint#1{\mathchoice
   {\XXint\displaystyle\textstyle{#1}}%
   {\XXint\textstyle\scriptstyle{#1}}%
   {\XXint\scriptstyle\scriptscriptstyle{#1}}%
   {\XXint\scriptscriptstyle\scriptscriptstyle{#1}}%
   \!\int}
\def\XXint#1#2#3{{\setbox0=\hbox{$#1{#2#3}{\int}$}
     \vcenter{\hbox{$#2#3$}}\kern-.5\wd0}}

\def\dashint{\Xint-}

\newcommand{\twopartdef}[4]
{
\left\{
		\begin{array}{ll}
			#1 & #2 \\
			#3 & #4
		\end{array}
	\right.
}

\definecolor{violet(ryb)}{rgb}{0.53, 0.0, 0.69}
\usepackage{mathtools}
\mathtoolsset{showonlyrefs}

\begin{document}

\title[1-Laplacian in Carnot groups]{\bf Local and nonlocal 1-Laplacian in Carnot groups}

\author[W. G\'orny]{Wojciech G\'orny}

\address{W. G\'orny: Faculty of Mathematics, Informatics and Mechanics, University of Warsaw, Banacha 2, 02-097 Warsaw, Poland.
\hfill\break\indent
{\tt w.gorny@mimuw.edu.pl }}


\setcounter{tocdepth}{1}

%
%\date{\today}

\begin{abstract}
We formulate and study the nonlocal and local least gradient problem, which is the Dirichlet problem for the 1-Laplace operator, in a quite natural setting of Carnot groups. We study the passage from the nonlocal problem to the local problem as the range of the interaction goes to zero; to do this, we first prove a total variation estimate of independent interest.
\end{abstract}

\maketitle

%{ \renewcommand\contentsname{Contents }
%\setcounter{tocdepth}{3}
%\tableofcontents
% }

\section{Introduction}

In the Euclidean setting, the least gradient problem is the following problem of minimisation 
\begin{equation}\tag{LGP}\label{problem}
\min \bigg\{ \int_\Omega |Du|: \quad u \in BV(\Omega), \quad u|_{\partial\Omega} = f \bigg\}.
\end{equation}
The problem in this form was introduced in \cite{SWZ}, but the ideas can be traced back to the pioneering work of Bombieri, de Giorgi and Giusti in \cite{BGG}, which relate the functions which locally minimise the total variation with minimal surfaces. This type of problems, including anisotropic cases, appears as dimensional reduction in the free material design (\cite{GRS2017NA}), conductivity imaging (\cite{JMN}), and has links to the optimal transport problem (\cite{DG2019,DS}).

In the paper \cite{MRL}, Maz\'on, Rossi and Segura de Le\'on showed that the 1-Laplace equation
\begin{equation}
- \mbox{div}\bigg(\frac{Du}{|Du|}\bigg) = 0
\end{equation}
can be understood as the Euler-Lagrange equation for the least gradient problem. The first step, namely convergence of solutions to the $p$-Laplace equations with continuous Dirichlet boundary data to the solutions of the corresponding least gradient problem as $p \rightarrow 1$, has been shown by Juutinen in \cite{Juu}; the authors of \cite{MRL} have provided an Euler-Lagrange characterisation of solutions to \eqref{problem} in terms of Anzelotti pairings and proved existence of solutions to such a problem using approximations by $p$-harmonic functions. In this paper, we follow this approach to the least gradient problem, but we work in the setting of Carnot groups.

The motivation for this paper comes from the nonlocal version of the least gradient problem. The nonlocal versions of 1-Laplace and $p$-Laplace equations, including evolution problems, have been studied for instance in \cite{AMRT1, AMRT,GM2019,MPRT,MST1}. The two main examples of nonlocal interaction are Markov chains on locally finite graphs and the interaction governed by a nonnegative radial kernel on $\mathbb{R}^n$. In the latter case, the nonlocal least gradient problem (the Dirichlet problem for the 1-Laplace operator) takes form
\begin{equation}\left\{\begin{array}{ll}
\D - \int_{\mathbb{R}^n} J(x-y) \frac{u_\psi(y) - u(x)}{|u_\psi(y) - u(x)|} \, dy = 0, &x\in\Omega,\\
u=\psi,& x\in \mathbb{R}^n \backslash \Omega,
\end{array}\right.
\end{equation}
where $\psi \in L^\infty(\mathbb{R}^n)$, $u_\psi$ is an extension of $u$ by $\psi$ outside of $\Omega$, and $J: \mathbb{R}^n \rightarrow [0, \infty)$ is a continuous radial function compactly supported in the unit ball such that $J(0) > 0$ and $\int_{\mathbb{R}^n} J dx = 1$. In \cite{MPRT}, the authors show existence of solutions to such a problem and show that equivalently we may minimise the nonlocal total variation as follows:
\begin{equation}
\min \bigg\{ \frac12 \int_{\mathbb{R}^n} \int_{\mathbb{R}^n} J(x-y) |u_\psi(y) - u_\psi(x)| \, dx \, dy: \quad u \in L^1(\Omega), \quad \psi \in L^1(\mathbb{R}^n) \bigg\}.
\end{equation}
A natural question is what happens if we rescale the kernel $J$ so that its support lies in a ball $B(0,\varepsilon)$ instead of the unit ball - are the nonlocal problems a good approximation to the local problem? In the Euclidean space, the authors of \cite{MPRT} show that the answer is yes and on a subsequence solutions to the nonlocal problems converge to a solution to the local problem.

In a recent paper \cite{GM2019} the authors introduce the least gradient problem in a nonlocal setting on metric measure spaces. Given a metric measure space $(X,d,\nu)$, the nonlocal interaction is governed by a random walk $m$, which is a collection of probability measures $m_x$ which are invariant and reversible with respect to the underlying measure $\nu$. The nonlocal total variation of a function $u \in L^1(X,\nu)$ in this case takes the form
$$ TV_m(u) = \frac12 \int_{X} \int_{X} |u(y) - u(x)| \, dm_x(y) \, d\nu(x).$$
Then, the authors introduce an Euler-Lagrange characterisation of the minimisers, which depends on whether the space supports a nonlocal Poincar\'e inequality for all $1 \leq p < \infty$. Then, the natural question is whether an analogue of the convergence mentioned above holds; namely, if we take the $\varepsilon$-step random walk (a natural generalisation of the random walk generated by the kernel $J$ to the setting of metric measure spaces)
$$ m_x = m_x^{\nu,\varepsilon} = \frac{\nu \res B(x,\varepsilon)}{\nu(B(x,\varepsilon))},$$
take solutions $u_\varepsilon$ to the nonlocal least gradient problem and pass to the limit $\varepsilon \rightarrow 0$, we want to know if the limit function $u$ (if it exists) is a solution to the local least gradient problem in some appropriate sense. 

Now, we come to the motivation why we study this problem in the setting of Carnot groups. The fact that $u_\varepsilon$ minimise the nonlocal total variation gives us a uniform bound on the nonlocal total variations of the form
\begin{equation}
\int_{X} \dashint_{B(x,\varepsilon)} |(u_\varepsilon)_\psi(y) - (u_\varepsilon)_\psi(x)| \, d\nu(y) \, d\nu(x) \leq M \varepsilon.
\end{equation}
Notice that the set on which we have this bound is the $\varepsilon$-neighbourhood of the diagonal in $X \times X$, in particular it changes with $\varepsilon$ and disappears in the limit $\varepsilon \rightarrow 0$. In order to pass to the limit, we will perform a blow-up argument, so that the domain remains the same. However, a blow-up argument requires a lot of structure of the metric space - we need to have well-defined directions (so that the difference quotients converge to something, i.e. directional derivatives) and the possibility to rescale balls in a controlled way. Therefore, it seems that we should work on a group which admits dilations which change the measure of balls in a controlled way. The prime example of such metric spaces are Carnot groups; on a Carnot group $\mathbb{G}$, which (in exponential coordinates) is the space $\mathbb{R}^n$ endowed with a group structure generated using the Baker-Campbell-Hausdorff formula from smooth horizontal vector fields $X_1, ..., X_m$, we have well-defined dilations, bounded domains satisfy a nonlocal Poincar\'e inequality for all $1 \leq p < \infty$, and the Lebesgue measure $\mathcal{L}^n$ is left-invariant and Ahlfors regular. Therefore, in this paper we will restrict our considerations to the setting of Carnot groups.

The structure of the paper is as follows: in Section \ref{sec:2preliminaries}, we recall the definitions of the Carnot groups, in particular the properties of exponential coordinates and trace and extension theorems. Then, we recall the formulations of the (local) least gradient problem in a Euclidean space and of the nonlocal least gradient problem in a metric random walk space. In Section \ref{sec:3lgpcarnot}, we introduce an Anzelotti-type pairing on Carnot groups, which allows us to formulate the (local) Dirichlet problem for the 1-Laplace equation on Carnot groups and study its properties.

In the final Section, we pass to the limit with $\varepsilon \rightarrow 0$. Firstly, in Theorem \ref{thm:amrtincarnotgroups} we show that a uniform bound on the nonlocal gradients of a given sequnce implies that the limit function lies in an appropriate Sobolev space or the BV space with respect to the Carnot group structure. Then, in Theorem \ref{thm:connectionwithlocal} we prove that on some subsequence solutions to the nonlocal problems converge weakly to the solutions of the local problem; in particular, this is an existence result for the local problem. Let us stress that the properties of Carnot groups are essential for the proofs presented here to work: we repeatedly use the group structure, left invariance of the Lebesgue measure, the fact that dilations behave differently in the horizontal directions than the other ones, and the symmetry of the balls in the box distance in horizontal directions.

\section{Preliminaries}\label{sec:2preliminaries}

\subsection{Carnot-Carath\'eodory distance}

In this subsection, we recall the notion of the Carnot-Carath\'eodory distance on $\mathbb{R}^n$ and $BV$ functions with respect to this distance (see for instance \cite{Vitt}). Later, we  restrict our focus to Carnot groups. The group structure will play a crucial part in the definition of Anzelotti pairings and in the proofs in Section \ref{sec:4nonlocalconvergence}.

\begin{definition}\label{dfn:ccdistance}
Let $X = (X_1, ..., X_m)$ be a family of smooth vector fields on $\mathbb{R}^n$ with $m \leq n$. 
We will say that $X_1,...,X_m$ are horizontal vector fields. We say that $\gamma: [0,T] \rightarrow \mathbb{R}^n$ is a subunit path, if
$$ \dot{\gamma}(t) = \sum_{j=1}^m h_j(t) X_j(\gamma(t)) \qquad \mbox{with} \qquad \sum_{j=1}^m h_j^2(t) \leq 1 \quad \mbox{for a.e. } t\in [0,T]$$
with $h_1,...,h_m$ measurable. We define the Carnot-Carath\'eodory (CC) distance on $\mathbb{R}^n$ as
\begin{equation*}
d_{cc}(x,y) = \inf \bigg\{ T \geq 0: \mbox{there is a subunit path } \gamma: [0,T] \rightarrow \mathbb{R}^n \mbox{ s.t. } \gamma(0)=x, \gamma(T)=y  \bigg\}.
\end{equation*}
Whenever the CC distance is finite for any pair $x,y \in  \mathbb{R}^n$, the space $\mathbb{R}^n$ endowed with the CC distance is a metric space. The space $(\mathbb{R}^n, d_{cc})$ is called a Carnot-Carath\'eodory space.
\end{definition}

In order to introduce $BV$ spaces with respect to the distance $d_{cc}$, we will assume the following connectivity condition
$$ d_{cc} \mbox{ is finite and the identity map } (\mathbb{R}^n, d_{cc}) \rightarrow (\mathbb{R}^n, |\cdot|) \mbox{ is a homeomorphism.}$$
There are a number of conditions which imply the above; as we will work primarily with Carnot groups, let us recall only the {\it Chow-H\"ormander condition}. Let $\mathfrak{L}(X_1,...,X_m)$ be the Lie algebra generated by the vector fields $X_1,...,X_m$. Then we require that
$$ \mbox{rank } \mathfrak{L}(X_1,...,X_m) = n.$$
On Carnot groups this condition is automatically satisfied, see Definitions \ref{dfn:liealgebrastratified}-\ref{dfn:carnotgroup}.

The most simple example of a Carnot-Carath\'eodory space is the Euclidean space $\mathbb{R}^n$ with the Euclidean distance arising from the family given by partial derivatives $(\partial_{x_1}, ..., \partial_{x_n})$; the following definitions then coincide with the classical definitions. A standard nontrivial example is the Heisenberg group described in Example \ref{ex:heisenberggroup}.

\begin{definition}\label{def:horizontalgradient}
Given a Lebesgue measurable function $u: \Omega \rightarrow \mathbb{R}$ defined on an open set $\Omega$, we define its {\it horizontal gradient} as
$$ Xu = (X_1 u, ..., X_m u),$$
where the derivatives are understood in the sense of distributions. We will say that $u \in C^1_X(\Omega)$, if both $u$ and $Xu$ are continuous, i.e. $u \in C(\Omega)$ and $Xu \in C(\Omega; \mathbb{R}^m)$.
\end{definition}

Then, we define the space of functions of bounded $X$-variation. Given a vector field $g = (g_1, ..., g_m) \in C_c^1(\Omega; \mathbb{R}^m)$, we define its $X$-divergence as
$$ \mbox{div}_X(g) = - \sum_{j = 1}^m X_j^* g_j,$$
where $X_j^*$ is the formal adjoint operator of $X_j$. In coordinates, if $X_j = \sum_{i=1}^n a_{ij}(x) \partial_i$, formally we have 
$$ X_j^* \psi(x) = - \sum_{i=1}^n \partial_i (a_{ij} \psi) (x).$$
On a Carnot group the adjoint operator is given by $X^* = -X$, see for instance \cite[Lemma 1.30]{VittPhD}; this is one of the reasons why we choose to work with Carnot groups in this paper. Now, we define the $X$-variation in a similar way as in the Euclidean case:

\begin{definition}
Let $\Omega \subset \mathbb{R}^n$ be open. We say that $u \in L^1(\Omega)$ has bounded {\it $X$-variation} in $\Omega$, if
$$ |Xu|(\Omega) = \sup \bigg\{ \int_\Omega u \, \mbox{div}_X(g) \, d\mathcal{L}^n: \quad g \in C_c^1(\Omega; \mathbb{R}^m), |g|_\infty \leq 1 \bigg\} $$
is finite. The space of functions with bounded $X$-variation in $\Omega$ is denoted by $BV_X(\Omega)$; endowed with the norm
$$ \| u \|_{BV_X(\Omega)} =  \| u \|_{L^1(\Omega)} + |Xu|(\Omega)$$
it is a Banach space. Moreover, $u \in BV_X(\Omega)$ if and only if $Xu$ is a (vectorial) Radon measure with finite total variation.
\end{definition}

The space $BV_X(\Omega)$ enjoys some properties of the Euclidean BV spaces such as lower semicontinuity of the total variation. We will focus mostly on traces of functions in $BV_X(\Omega)$: in order to perform the construction of Anzelotti pairings, we firstly define certain objects for smooth functions and have to approximate a given function $u \in BV_X(\Omega)$ in such a way that the trace of the limit function is preserved; for this approximation in the Euclidean case, see \cite[Lemma 5.2]{Anz}. To this end, we recall a few results concerning trace theory in Carnot-Carath\'eodory spaces, following \cite{Vitt} (see also \cite{MM}).

For a set $E \subset \mathbb{R}^n$ we define the $X$-perimeter measure $|\partial\Omega|_X$ as the $X$-variation of its characteristic function $\chi_E$. Suppose that $E$ has finite $X$-perimeter in $\Omega$ and denote by $\nu_E: \mathbb{R}^n \rightarrow S^{m-1}$ the density of $X\chi_E$ with respect to $|X\chi_E|$, i.e. $X\chi_E = \nu_E |X\chi_E|$ as measures. Then
$$ \int_E \mbox{div}_X g \, d\mathcal{L}^n = - \int_\Omega g \cdot \nu_E \, d|\partial\Omega|_X.$$
We call such $\nu_E$ the {\it horizontal inner normal} to $E$. 

The following theorem proved in \cite[Theorem 1.4]{Vitt} asserts the existence of traces of functions in $BV_X(\Omega)$ and shows that their traces lie in $L^1(\partial\Omega, |\partial\Omega|_X)$. 

\begin{theorem}\label{thm:vittonetraces}
Let $\Omega \subset \mathbb{R}^n$ be an $X$-Lipschitz domain with compact boundary. Then, there exists a bounded linear operator
$$ T: BV_X(\Omega) \rightarrow L^1(\partial\Omega, |\partial\Omega|_X) $$
such that
\begin{equation}
\int_\Omega u \, \mbox{div}_X g \, d\mathcal{L}^n + \int_\Omega g \, d Xu = \int_{\partial\Omega} g \cdot \nu_\Omega \, Tu \, d|\partial\Omega|_X
\end{equation}
for any $u \in BV_X(\Omega)$ and $g \in C^1(\mathbb{R}^n, \mathbb{R}^m)$. Here, $\nu_\Omega$ is the horizontal inner normal to $\Omega$. 
\end{theorem}

The trace operator introduced in the above theorem enables us to consider the restrictions of the horizontal gradient $Xu$ to subsets of codimension one, such as the boundary of the domain $\Omega$. Given an $X$-Lipschitz domain $\Omega$ with compact boundary, denote by $T^+$ the trace operator $T^+: BV_X(\Omega) \rightarrow L^1(\partial\Omega, |\partial\Omega|_X)$ and by $T^-$ the trace operator $T^-: BV_X(\mathbb{R}^n \backslash \overline{\Omega}) \rightarrow L^1(\partial\Omega, |\partial\Omega|_X)$. The following result was shown in \cite[Theorem 5.3]{Vitt}.

\begin{theorem}\label{thm:vittonesplittingthegradient}
Let $\Omega \subset \mathbb{R}^n$ be an $X$-Lipschitz domain with compact boundary. Suppose that $u \in L^1(\mathbb{R}^n)$ is such that $u|_{\Omega} \in BV_X(\Omega)$ and $u|_{\mathbb{R}^n \backslash \overline{\Omega}} \in BV_X(\mathbb{R}^n \backslash \overline{\Omega})$. Then $u \in BV_X(\mathbb{R}^n)$ and
$$ Xu = Xu \res \Omega + Xu \res (\mathbb{R}^n \backslash \overline{\Omega}) + (T^+ u - T^- u) \nu_\Omega |\partial\Omega|_X.$$
\end{theorem}

The next theorem proved in \cite[Theorem 1.5]{Vitt} concerns extensions of functions in $L^1(\partial\Omega, |\partial\Omega|_X)$ to $BV_X(\Omega)$. In particular,  $L^1(\partial\Omega, |\partial\Omega|_X)$ is precisely the trace space of $BV_X(\Omega)$.

\begin{theorem}\label{thm:vittoneextension}
Let $\Omega \subset \mathbb{R}^n$ be a $X$-Lipschitz domain with compact boundary. Then, there exists $C(\Omega)$ with the following property: for any $h \in L^1(\partial\Omega, |\partial\Omega|_X)$ and $\varepsilon > 0$, there exists $u \in C^\infty(\Omega) \cap W^{1,1}_X(\Omega)$ such that
\begin{equation}
Tu = h, \quad \int_\Omega |u| \, d\mathcal{L}^n \leq \varepsilon \quad \mbox{and} \quad \int_\Omega |Xu| \, d\mathcal{L}^n \leq C(\Omega) \| h \|_{L^1(\partial\Omega, |\partial\Omega|)}.
\end{equation}
If $\partial\Omega$ is $X$-regular, then $u$ can be chosen in such a way that
\begin{equation}
\int_\Omega |Xu| \, d\mathcal{L}^n \leq (1 + \varepsilon) \| h \|_{L^1(\partial\Omega, |\partial\Omega|_X)}.     
\end{equation}
Furthermore, as we can see from the proof of \cite[Theorem 1.5]{Vitt}, we may require two more things: firstly, we can ensure that the support of $u$ lies in an arbitrarily small neighbourhood of $\partial\Omega$ and require that
\begin{equation}
u(x) = 0 \quad \mbox{if} \quad \mbox{dist}(x,\partial\Omega) > \varepsilon.
\end{equation}
Moreover, if additionally $h \in L^\infty(\partial\Omega,|\partial\Omega|_X)$, then we may require that
$$ \| u \|_{L^\infty(\Omega)} \leq \| h \|_{L^\infty(\partial\Omega,|\partial\Omega|_X)}.$$
\end{theorem}

We turn our focus to approximations of functions in $BV_X(\Omega)$ by smooth functions.

\begin{definition}
We say that $u_k \in BV_X(\Omega)$ converges $X$-strictly to $u \in BV_X(\Omega)$, if
$$ u_k \rightarrow u \quad \mbox{in } L^1(\Omega) \qquad \mbox{and} \qquad \int_\Omega |Xu_k| \rightarrow \int_\Omega |Xu|.$$
\end{definition}

As in the Euclidean case, $X$-strict convergence is a natural requirement for working with approximations of $BV$ functions. Usually, the norm convergence is too strong a requirement to ask; on the other hand, weak* convergence does not entail convergence of traces of the approximating sequence to the trace of the limit. The next two results, proved in \cite[Corollary 5.5]{Vitt} and \cite[Theorem 5.6]{Vitt} respectively, are generalizations of well-known results in the Euclidean case.

\begin{proposition}\label{prop:vittoneapproximation}
Let $\Omega \subset \mathbb{R}^n$ be an $X$-Lipschitz domain with compact boundary and $u \in BV_X(\Omega)$. Then, there exists a sequence $u_k \in C^\infty(\Omega) \cap C^0(\overline{\Omega}) \cap BV_X(\Omega)$ which converges $X$-strictly to $u$. Furthermore, as we can see from the proof of \cite[Corollary 5.5]{Vitt}, if additionally $u \in L^\infty(\Omega)$, then we may require that
$$ \| u_k \|_{L^\infty(\Omega)} \leq \| u \|_{L^\infty(\Omega)}.$$
\end{proposition}

\begin{theorem}\label{thm:vittonecontinuity}
Let $\Omega \subset \mathbb{R}^n$ be an $X$-Lipschitz domain with compact boundary. Suppose that the sequence $u_k \in BV_X(\Omega)$ converges $X$-strictly to $u \in BV_X(\Omega)$. Then $Tu_k \rightarrow Tu$ in $L^1(\partial\Omega, |\partial\Omega|_X)$.
\end{theorem}

In order to introduce Anzelotti pairings on Carnot groups, we are going to require something more than Proposition \ref{prop:vittoneapproximation} - we need the approximating sequence to preserve the trace of the limit. Fortunately, it is an easy consequence of the results above. For a similar result in the Euclidean case, see \cite[Lemma 5.2]{Anz}.

\begin{proposition}\label{prop:approximationpreservingtraces}
Let $\Omega \subset \mathbb{R}^n$ be an $X$-Lipschitz domain with compact boundary and $u \in BV_X(\Omega)$. Then, there exists a sequence $u_k \in C^\infty(\Omega) \cap BV_X(\Omega)$ which converges $X$-strictly to $u$ and such that $Tu_k = Tu$. If additionally $u \in L^\infty(\Omega)$, then also
$$ \| u_k \|_{L^\infty(\Omega)} \leq 3 \| u \|_{L^\infty(\Omega)}.$$
\end{proposition}

\begin{proof}
Take an approximating sequence $u_k \in C^\infty(\Omega) \cap BV_X(\Omega)$ converging $X$-strictly to $u$ given by Proposition \ref{prop:vittoneapproximation}. Denote
$$ h_k = Tu - Tu_k \in L^1(\partial\Omega, |\partial\Omega|_X).$$
Set $\delta = \frac{1}{k}$ and take an extension $v_k \in C^\infty(\Omega) \cap BV_X(\Omega)$ of $h_k$ given by Theorem \ref{thm:vittoneextension}. Set
$$\widetilde{u_k} = u_k + v_k \in C^\infty(\Omega) \cap BV_X(\Omega).$$
Then $\widetilde{u_k}$ is the desired approximating sequence. Firstly, by linearity of the trace operator
$$T\widetilde{u_k} = Tu_k + Tu - Tu_k = Tu.$$
Secondly, as $\int_\Omega |v_k| \, d\mathcal{L}^n \leq \frac1k$, we have
$$ \int_{\Omega} |\widetilde{u_k} - u| \, d\mathcal{L}^n \leq \int_{\Omega} |u_k - u| \, d\mathcal{L}^n + \int_{\Omega} |v_k| \, d\mathcal{L}^n \leq \int_{\Omega} |u_k - u| \, d\mathcal{L}^n + \frac1k \rightarrow 0.$$
Thirdly, as $\int_\Omega |Xv_k| \, d\mathcal{L}^n \leq C(\Omega) \int_{\partial\Omega} |h_k| \, d|\partial\Omega|_X$, we have
$$ \int_{\Omega} |X\widetilde{u_k}| \, d\mathcal{L}^n \leq \int_{\Omega} |Xu_k| \, d\mathcal{L}^n + \int_{\Omega} |Xv_k| \, d\mathcal{L}^n \leq \int_{\Omega} |Xu_k| \, d\mathcal{L}^n + C(\Omega) \int_{\partial\Omega} |h_k| \, d|\partial\Omega|_X = $$
$$ = \int_{\Omega} |Xu_k| \, d\mathcal{L}^n + C(\Omega) \int_{\partial\Omega} |Tu - Tu_k| \, d|\partial\Omega|_X \rightarrow \int_\Omega |Xu|,$$
where the second summand goes to zero by Theorem \ref{thm:vittonecontinuity}. Finally, if $u \in L^\infty(\Omega)$, then
$$ \| \widetilde{u_k} \|_{L^\infty(\Omega)} \leq \| u_k \|_{L^\infty(\Omega)} + \| v_k \|_{L^\infty(\Omega)} \leq \| u \|_{L^\infty(\Omega)} + \| h_k \|_{L^\infty(\partial\Omega, |\partial\Omega|_X)} \leq $$
$$ \leq \| u \|_{L^\infty(\Omega)} + \| Tu \|_{L^\infty(\partial\Omega, |\partial\Omega|_X)} + \| Tu_k \|_{L^\infty(\partial\Omega, |\partial\Omega|_X)} \leq $$
$$ \leq 2 \| u \|_{L^\infty(\Omega)} + \| u_k \|_{L^\infty(\Omega)} \leq 3 \| u \|_{L^\infty(\Omega)}.$$
The $L^\infty$ bound is proved here in a crude way and is clearly suboptimal, but we are not interested in the exact bound, as it is sufficient for the definition of Anzelotti pairings.
\end{proof}

\subsection{Definition of Carnot groups}

\begin{definition}\label{dfn:liealgebrastratified}
We say that a Lie algebra $\mathfrak{g}$ is stratified, if there exist linear subspaces $\mathfrak{g}_1,...,\mathfrak{g}_l$ such that
$$ \mathfrak{g} = \mathfrak{g}_1 \oplus ... \oplus \mathfrak{g}_l$$
and such that the following condition holds:
$$ \mathfrak{g}_{j} = [\mathfrak{g}_1, \mathfrak{g}_{j-1}] \quad \mbox{for } j=2,...,l \qquad \mbox{and } [\mathfrak{g}_1,\mathfrak{g}_l]=\{ 0 \}.$$
We call this decomposition a stratification of $\mathfrak{g}$. We call the elements of $\mathfrak{g}_1$ the horizontal vector fields. 
\end{definition}

\begin{definition}
We say that a Lie group $\mathbb{G}$ is stratified if its Lie algebra is stratified. If the group $\mathbb{G}$ is finite dimensional and stratified, then it is also nilpotent of step $l$.
\end{definition}

\begin{definition}\label{dfn:carnotgroup}
A Carnot group is a finite dimensional, connected, simply connected and stratified Lie group $\mathbb{G}$ (of step $l$). 
\end{definition}

On Carnot groups, the exponential map $\exp: \mathfrak{g} \rightarrow \mathbb{G}$ is a (global) diffeomorphism. This allows for a definition of dilations on the Carnot group and the introduction of exponential coordinates.

\begin{definition}
On a stratified Lie algebra $\mathfrak{g}$, we define a one-parameter group of dilations of the algebra by the formula
$$ \delta_\lambda X = \lambda^j X \qquad \mbox{if } X \in \mathfrak{g}_j$$
and extend it to the whole of $\mathfrak{g}$ by linearity. As the exponential map is a global diffeomorphism, we extend it to the Lie group by the formula
$$ \delta_\lambda(x) = \exp(\delta_\lambda(\exp^{-1}(x)))).$$
\end{definition}

Now, we introduce exponential coordinates. For vector fields $X,Y \in \mathfrak{g}$, we define $C(X,Y)$ by the formula
$$ \exp(C(X,Y)) = \exp(X)\exp(Y).$$
In fact, there is a direct formula for $C(X,Y)$ called the Baker-Campbell-Hausdorff formula; it is formally an infinite series (not necessarily convergent) defined by iterated commutators of $X,Y$. We recall the first few summands in the Baker-Campbell-Hausdorff formula:
$$ C(X,Y) = X + Y + \frac12 [X,Y] + \frac1{12} \bigg([X,[X,Y]] - [Y, [X,Y]] \bigg) + ...$$
The next summands involve iterations of an increasing number of commutators of $X,Y$. Due to the stratified structure of $\mathfrak{g}$ in the case of Carnot groups the BCH formula is a polynomial which converges everywhere (as any iterations of $l+1$ or more commutators are zero). Moreover, we see that it is linear in the horizontal directions (in $\mathfrak{g}_1$).

\begin{definition}\label{dfn:gradedcoordinates}
Let $(X_1, ..., X_n)$ be a basis of the Lie algebra $\mathfrak{g}$ of left invariant vector fields. We say that the basis $(X_1,...,X_n)$ of $\mathfrak{g}$ is adapted to the stratification, if it is ordered in the same way as $\mathfrak{g}_j$; precisely, let $m_j = \mbox{dim}(\mathfrak{g}_j)$ and $n_j = m_1 + ... + m_j$. Then, we require that $X_1,...,X_{n_1}$ is a basis of $\mathfrak{g}_1$ and for $j > 1$ we have that $X_{n_j -1},...,X_{n_j}$ is a basis of $\mathfrak{g}_j$. 
\end{definition}

\begin{definition}
A system of exponential coordinates (of the first kind) relative to a basis $(X_1, ..., X_n)$ of $\mathfrak{g}$ adapted to the stratification is a map from $\mathbb{R}^n$ to $\mathbb{G}$ defined by
$$ x \mapsto \exp \bigg( \sum_{i=1}^n x_j X_j \bigg).$$
We endow $\mathbb{R}^n$ with the group operation pulled back from $\mathbb{G}$, i.e.
$$ x \circ y = z \quad \Leftrightarrow \quad \sum_{i=1} z_j X_j = C(\sum_{i=1} x_j X_j, \sum_{i=1} y_j X_j).$$
We see that $\mathbb{R}^n$ with this group law is a Lie group with a Lie algebra isomorphic to $\mathfrak{g}$. As both $\mathbb{G}$ and $\mathbb{R}^n$ are nilpotent, connected and simply connected, the exponential coordinates define a diffeomorphism between $\mathbb{R}^n$ and $\mathbb{G}$. We endow $\mathbb{R}^n$ with the Carnot-Carath\'eodory distance introduced in Definition \ref{dfn:ccdistance}.
\end{definition}

{\bf Notation.} In this paper, we will exploit the structure of $\mathbb{G}$ in exponential coordinates and use a version of the scalar product on $\mathbb{R}^n$ restricted to horizontal directions. Given two vectors $v \in \mathbb{R}^{n_1}$ and $w \in \mathbb{R}^{n_2}$, where $m \leq n_1, n_2 \leq n$ with $n_1 \neq n_2$, we set
$$ \langle v, w \rangle = \sum_{i=1}^m v_i w_i.$$
The only difference with respect to the usual scalar product is that the dimensions do not coincide. In other words, we extend $v$ and $w$ to $\mathbb{R}^n$ by setting it to zero on the last $n-n_1$ (respectively $n - n_2$) coordinates and then we take the usual scalar product. When the dimensions $n_1, n_2$ coincide, we will denote the scalar product by $v \cdot w$, so that the notation $\langle v, w \rangle$ is used to alert the reader that the dimensions are different. Finally, we denote the group operation by $\circ$ or skip the multiplication symbol when it is clear from the context.

In the next Proposition, we list a few well-known properties of exponential coordinates; we refer for instance to \cite{VittPhD}.

\begin{proposition}\label{prop:propertiesexponentialcoordinates}
(properties of exponential coordinates) \\
1. In this representation, the neutral element of $\mathbb{G}$ is $(0,...,0)$. Moreover, for all $x \in \mathbb{R}^n$ we have $x^{-1} = -x$. \\
2. As the basis $(X_1,...,X_n)$ of $\mathfrak{g}$ is adapted to the stratification, the dilations $\delta_\lambda$ in the exponential coordinates are represented as
$$ \delta_\lambda(x_1,...,x_n) = (\lambda x_1,..., \lambda x_{n_1}, \lambda^2 x_{n_1 + 1}, ..., \lambda^2 x_{n_2},..., \lambda^l x_{n_{l-1}+1}, \lambda^l x_n).$$
3. By the Baker-Campbell-Hausdorff formula, the group operation is linear in the horizontal directions (directions generated from $\mathfrak{g}_1$). \\
4. The length of horizontal curves is preserved by left translation in $\mathbb{G}$. In particular, we have $d_{cc}(z \circ x, z \circ y) = d_{cc}(x,y)$. \\
5. The CC distance is $1$-homogeneous with respect to the dilations $\delta_\lambda$, i.e. $d_{cc}(\delta_\lambda x, \delta_\lambda y) = \lambda d_{cc}(x,y)$. \\
6. In the exponential coordinates the Lebesgue measure $\mathcal{L}^n$ is the Haar measure of $\mathbb{G}$ and is both left- and right-invariant. \\
7. $\mathcal{L}^n$ is Ahlfors regular with respect to $d_{cc}$ with exponent $Q = \sum_{j=1}^l j \, \dim(\mathfrak{g}_j)$, called the homogeneous dimension of $\mathbb{G}$. Precisely, we have
$$ \mathcal{L}^n(B_{cc}(x,r)) = \mathcal{L}^n(B_{cc}(0,r)) = r^Q \mathcal{L}^n(B_{cc}(0,1)).$$
Note that the exponent $Q$ comes from the homogeneity of $d_{cc}$ with respect to dilations and the change of variables formula.
\end{proposition}

In exponential coordinates we have another distance with similar properties as the Carnot-Carath\'eodory distance: the box distance, which due to its simple form is often more convenient to use that the CC distance.

\begin{definition}
In the notation of Definition \ref{dfn:gradedcoordinates}, take the norm on $\mathbb{R}^n$ defined by
\begin{equation}
\| (x_1,...,x_n) \| := \max \bigg\{ |x_1|, ..., |x_{n_1}|, |x_{n_1+1}|^{\frac12}, ..., |x_{n_2}|^{\frac12}, ..., |x_{n_{k-1}}+1|^{\frac1k}, ..., |x_{n_k}|^{\frac1k} \bigg\}.
\end{equation}
For $x,y \in \mathbb{G}$, we define the box distance by the formula
$$d_{box}(x,y) := \| x^{-1} y \|.$$
We will denote balls with respect to the box distance by $U(x,r)$.
\end{definition}

In general, one may consider distances that are left-invariant and $1$-homogenous with respect to the dilations; such distances are for simplicity called homogenous (\cite{LeD}). It is clear that the box distance is homogenous and that balls with respect to homogenous distances have the properties from point $7$ in Proposition \ref{prop:propertiesexponentialcoordinates}. In particular, we have
$$ \mathcal{L}^n(U(x,r)) = \mathcal{L}^n(U(0,r)) = r^Q \mathcal{L}^n(U(0,1)).$$
We choose to work primarily with the box distance due to the fact that the unit ball with respect to this distance is symmetric in the horizontal directions; this will play a key role in the proofs in Section \ref{sec:4nonlocalconvergence}. This is known for the CC distance only in a few cases, such as Heisenberg groups, see \cite{Mon00} or \cite[Corollary 3.2]{HZ}.

As we can see from the Baker-Campbell-Hausdorff formula, in the exponential coordinates the only abelian Carnot group of dimension $n$ is $\mathbb{R}^n$ with the group action being simply addition. A more typical example is the (non-abelian) Heisenberg group.

\begin{example}\label{ex:heisenberggroup}
(Heisenberg group) \\
The Heisenberg group $\mathbb{H}^1$ is the space $\mathbb{R}^3$ equipped with the following vector fields:
$$ X_1 = \partial_1 - \frac12 x_2 \cdot \partial_3, \quad X_2 = \partial_2 + \frac12 x_1 \cdot \partial_3, \quad X_3 = [X_1, X_2] = \partial_3. $$
In particular, $\mathfrak{g}_1 = \mbox{span}(X_1, X_2)$ and $\mathfrak{g}_2 = \mbox{span}(X_3)$; in other words, the vector fields $X_1, X_2$ are the basis of the subspace of horizontal vector fields. \\
In exponential coordinates, the group structure on the Heisenberg group induced by these (left-invariant) vector fields is as follows:
$$ (x_1,x_2,x_3) \circ (x_1', x_2', x_3') = (x_1 + x_1', x_2 + x_2', x_3 + x_3' + \frac12 (x_1 x_2' - x_2 x_1')).$$
In particular, we see that the inverse of $(x_1, x_2, x_3)$ is $(-x_1, -x_2, -x_3)$. Moreover, the Heisenberg group is equipped with the dilations $\delta_\lambda: \mathbb{H}^1 \rightarrow \mathbb{H}^1$ defined by the formula
$$ \delta_\lambda((x_1, x_2, x_3)) = (\lambda x_1, \lambda x_2, \lambda^2 x_3).$$
The space $\mathbb{H}^1$ is equipped with the Carnot-Carath\'eodory distance generated by the horizontal vector fields $X_1, X_2$.
The measure $\mathcal{L}^3$ is the Haar measure of this group. Moreover, if we calculate the Jacobian of the dilation $\delta_\lambda$, we see that dilations rescale the measure $\mathcal{L}^3$ by a factor of $\lambda^4$. In particular, we have that
$$ \mathcal{L}^3(U(x,r)) = \mathcal{L}^3(U(0,r)) = \mathcal{L}^3(U(0,1)) \, r^4.$$
\end{example}

%%%%%%%%%%%%%%%%%%%%%%%%%%%%%%%%%%%%%%

\subsection{Local least gradient problem in $\mathbb{R}^n$}

Let $\Omega \subset \R^n$ be a bounded open set. The least gradient problem involves minimisation of the relaxed total variation functional relative to Dirichlet boundary data, i.e. the functional $\Phi_h: L^{\frac{n}{n-1}}(\Omega)
\rightarrow (-\infty,+\infty]$ defined by 
\begin{equation}\label{Def_Phi_PDE}
\Phi_h(u) = \left\{ \begin{array}{ll} \displaystyle
\int_{\Omega}\vert Du \vert + \int_{\partial \Omega} \vert u - h \vert\, d \mathcal H^{n-1} & \hbox{if} \ u \in BV(\Omega),
\\ +\infty & \hbox{if} \  u \in  L^{\frac{n}{n-1}}(\Omega) \setminus BV(\Omega).
\end{array}\right.
\end{equation}
An alternative formulation, in the language of Euler-Lagrange equations, has been introduced in \cite{MRL}. The authors consider the Dirichlet problem for the $1$-Laplace operator
\begin{equation}\label{ELe}
    \left\{\begin{array}{ll}
    \displaystyle -\mbox{div}\Big(\frac{Du}{|Du|}\Big)=0\,,&\hbox{in }\Omega\,,\\[10pt]
    u=h\,,&\hbox{on }\partial\Omega\,,
    \end{array}\right.
\end{equation}
and show that it has a solution $u \in BV(\Omega)$  for every $h\in
L^1(\partial\Omega)$. The notion of a solution is introduced in the language of Anzelotti pairings; for the definition of Anzelotti pairings in Euclidean spaces, we refer to \cite{Anz}. We skip the definition in this introduction as we need to prove existence of a similar pairing for Carnot groups in Section \ref{sec:3lgpcarnot}.

\begin{definition}\label{dfn:1laplaceeuclidean}
We say that $u \in BV(\Omega)$ is a solution of the 1-Laplace equation for boundary data $h \in L^1(\partial\Omega, \mathcal{H}^{n-1})$, if there exists a vector field $\mathbf{z} \in L^\infty(\Omega; \mathbb{R}^n)$ such that $\| \mathbf{z} \|_{L^\infty(\Omega; \mathbb{R}^n)} \leq 1$ and the following conditions hold:
\begin{equation}
\mbox{div}(\mathbf{z}) = 0 \qquad \mbox{as distributions in }\Omega ;
\end{equation}
\begin{equation}
(\mathbf{z}, Du) = |Du| \qquad \mbox{as measures in }\Omega;
\end{equation}
\begin{equation}
[\mathbf{z} \cdot \nu] \in \mbox{sign}(h-u) \qquad \mathcal{H}^{n-1}-\mbox{a.e. on }\partial\Omega.
\end{equation}
\end{definition}
In \cite{MRL}, the authors prove that the two approaches are equivalent - minimizers of the functional $\Phi_h$ coincide with the solutions of problem \eqref{ELe}.

%%%%%%%%%%%%%%%%%%%%%%%%%%%%%%%%%%%%%%

\subsection{Nonlocal least gradient problem in a metric measure space}

In this subsection, we recall the notion of a solution to the nonlocal least gradient problem on a metric measure space. In the final Section of this paper, we will connect the nonlocal and local versions of the least gradient problem in the setting of Carnot groups.

\begin{definition}
Let $(X,d,\nu)$ be a metric measure space. A metric random walk $m$ is a family of probability measures $m_x$ which satisfies two conditions: \\
(i) the dependence on $x$ is borelian, namely for Borel sets $A \subset X$ and $B \subset \mathbb{R}$ we have that $\{ x \in X: m_x(A) \in B\}$ is Borel; \\
(ii) each $m_x$ has finite first moment, i.e. for any $z \in X$ we have $\int_X d(z,y) \, dm_x(y) < \infty$. \\
We say that the measure $\nu$ is invariant (with respect to a random walk $m$), if for any $\nu$-measurable set $A$ and $\nu-$almost all $x \in X$ the map $x \mapsto m_x(A)$ is $\nu$-measurable and
$$ \nu(A) = \int_X m_x(A) d\nu(x).$$
We say that the measure $\nu$ is reversible, if a more detailed balance condition holds:
$$ dm_x(y) d\nu(x) = dm_y(x) d\nu(y).$$
\end{definition}

We present two examples of random walks. The first one has been studied extensively, for instance in \cite{AMRT}, \cite{BBM} and \cite{MRT}, as a prototype of nonlocal interaction on $\mathbb{R}^n$. The second one is its analogue which can be defined in a more general setting and which will be the focus of Section \ref{sec:4nonlocalconvergence}.

\begin{example}
(1) Let $X = \mathbb{R}^n$ with the Euclidean distance and let $\nu = \mathcal{L}^n$. Let $J: \mathbb{R}^n \rightarrow [0,\infty)$ be a measurable, nonnegative and radially symmetric function such that $\int_{\mathbb{R}^n} J(z) dz = 1$. Then, we define the random walk $m^J$ by the formula
$$ m^J_x(A) = \int_A J(x-y) \, d\mathcal{L}^n(y) \qquad \mbox{for every Borel set } A.$$
(2) Let $(X,d,\nu)$ be a metric measure space. Assume that the measure of balls in $X$ is finite and that $\mbox{supp } \nu = X$. Given $\varepsilon > 0$, we define the random walk $m^{\nu,\varepsilon}$ by the formula
$$ m_x^{\nu, \varepsilon} = \frac{\nu \res B(x,\varepsilon)}{\nu(B(x,\varepsilon))}.$$
The metric random walks defined in (1) and (2) are invariant and reversible.
\end{example} 

For an open set $\Omega \subset X$, denote by $\Omega_m$ the set
$$ \Omega_m = \{ x \in X: m_x(\Omega) > 0 \}.$$
Moreover, we denote
$$ \partial_m \Omega = \Omega_m \backslash \Omega$$
and consider open sets $\Omega$ such that 
\begin{equation}\label{eq:positiveboundarymeasure}
0 < \nu(\Omega) < \nu(\Omega_m) < \nu(X).  
\end{equation}

\begin{definition} We say that $(m,\nu)$ satisfies a $q$-Poincar\'{e} Inequality, $q \geq 1$, if
\begin{equation}\label{Poincareq}
\lambda\int_\Omega \left|u(x) \right|^q \, d\nu(x)\le \int_{\Omega}
\int_{\Omega_m}  |u_{\psi}(y)-u(x)|^q \, dm_x(y) \,
d\nu(x)+\int_{\partial_m\Omega}|\psi (y)|^q \, d\nu(y)
 \end{equation}
for all $u\in L^q(\Omega, \nu)$.
\end{definition}

The nonlocal least gradient problem is the problem of minimisation of the nonlocal total variation with respect to Dirichlet boundary condition. Given $\psi \in L^1(\Omega_m)$, we set
$$ u_\psi (x) = \twopartdef{u(x)}{x \in \Omega,}{\psi(x)}{x \in \partial_m\Omega}$$
and we minimise the relaxed total variation functional
\begin{equation}\label{eq:functionaljpsi}
\mathcal{J}_\psi(u) = \frac12 \int_{\Omega_m} \int_{\Omega_m} |u_\psi(y) - u_\psi(x)| \, dm_x(y) \, d\nu(x).
\end{equation} 
An equivalent formulation, of the Euler-Lagrange type, has been proved in \cite{GM2019} (see also \cite{MPRT} for the case of the random walk $m^J$). On a metric measure space $(X,d,\mu)$ equipped with a random walk $m$, assuming that $\nu$ is invariant and reversible, the authors of \cite{MST1} introduced the $m-1$-Laplacian operator $\Delta^m_1$, which formally is the operator
$$\Delta^m_1 u(x) = \int_{\Omega_m}
\frac{u_\psi(y)-u(x)}{|u_\psi(y)-u(x)|}dm_x(y).$$
In the paper \cite{GM2019} the authors consider the {\it nonlocal
$1$-Laplacian problem} with Dirichlet boundary condition~$\psi$:
\begin{equation}\label{1700.intro}
\left\{\begin{array} {ll} \displaystyle  -\Delta^m_1 u(x)=0,& x\in\Omega,
\\[10pt] u(x)=\psi(x),& x\in \partial_m\Omega.
\end{array}\right.
\end{equation}
Note that due to the assumption \eqref{eq:positiveboundarymeasure} the boundary condition is well-defined, since $\nu(\partial_m \Omega) > 0$. Now, we introduce the Euler-Lagrange equations for the functional $\mathcal{J}_\psi$.

\begin{definition}\label{Defi.1.var} {\rm Let $\psi\in L^1(\partial_m\Omega)$.
We say that $u\in BV_m(\Omega)$ is a    solution  to
\eqref{1700.intro}  if there exists $g\in L^\infty(\Omega_m\times \Omega_m)$ with
$\|g\|_\infty\leq 1$ verifying
\begin{equation}\label{lasim}
 g(x,y)=-g(y,x)\quad\hbox{for  $(\nu\otimes dm_x)$-a.e $(x,y)$ \ in $\Omega_m\times \Omega_m$},
\end{equation}
\begin{equation}\label{1-lapla2.var}
g(x,y)\in \mbox{sign}(u_\psi(y)-u_\psi(x))\quad
\hbox{for  $(\nu\otimes dm_x)$-a.e $(x,y)$ \ in $\Omega_m\times \Omega_m$},
\end{equation} and
\begin{equation}\label{1-lapla.var}
-\int_{\Omega_m}g(x,y)\,dm_x(y)= 0 \quad \nu-\mbox{a.e }x\in \Omega.
\end{equation}
The two notions are equivalent if we assume that the domain $\Omega$ supports a nonlocal $p$-Poincar\'e inequality for all $p > 1$. The following result was proved in \cite[Theorem 2.8]{GM2019}.}
\end{definition}

\begin{theorem} \label{teo.varia}   
Let $\psi\in L^\infty(\partial_m\Omega)$.  Suppose that $\Omega$ supports a $p$-Poincar\'{e} inequality for all $p > 1$. Then, there exists a solution $u \in L^1(\Omega)$ to \eqref{1700.intro}. Moreover, any $u\in L^1(\Omega)$ is a solution to \eqref{1700.intro} if and only if it is a minimiser of the functional $\mathcal{J}_{\psi}$  given in \eqref{eq:functionaljpsi}.
\end{theorem}

A key point is that since Carnot groups are length spaces equipped with a doubling measure, by \cite[Proposition 2.16]{GM2019} the Poincar\'e inequality holds for all $p > 1$ for bounded open sets in Carnot groups equipped with the random walk $m^{\nu,\varepsilon}$. In particular, there exist solutions to the nonlocal least gradient problem. A natural question is what happens if we let $\varepsilon \rightarrow 0$; in \cite{MPRT} the authors proved that in the Euclidean space for the random walk $m^J$ the solutions to the nonlocal problem converge on a subsequence to a solution of the local problem. We will explore this issue on Carnot groups in Section \ref{sec:4nonlocalconvergence}.

\section{Least gradient problem in Carnot groups}\label{sec:3lgpcarnot}

In this Section, we begin by introducing the notion of Anzelotti-type pairings on Carnot groups. The first step is existence of a weak normal trace of a divergence-measure vector field on the boundary of a sufficiently regular set, which will later play a role as the right notion of boundary values of a solution to the least gradient problem. We follow the outline presented by Anzelotti in \cite{Anz}; for basic facts about BV functions in CC spaces we refer to \cite{Vitt}. The second step will be a construction of a pairing $(\mathbf{z}, Xu)$ between a vector field and a BV function, which enables an analogue of Green's formula to hold for general BV functions. Finally, we introduce the notion of solution of the least gradient problem on Carnot groups, using the Anzelotti pairings introduced above.

\subsection{Anzelotti-type pairings on Carnot groups}

In this subsection, our goal is to introduce an analogue of Anzelotti pairings on Carnot groups. Let us introduce the following spaces:
$$ DM_X(\Omega) = \{ \mathbf{z} \in L^\infty(\Omega; \mathbb{R}^m): \, \mbox{div}_X(\mathbf{z}) \mbox{ is a finite measure in } \Omega\},$$
and
$$ BV_X^c(\Omega) = BV_X(\Omega) \cap L^\infty(\Omega) \cap C(\Omega).$$
We want to construct the weak normal trace on the boundary of a sufficiently smooth set. We follow the outline of \cite{Anz}. Firstly, we will introduce an auxiliary pairing
$$ \langle \mathbf{z}, u \rangle_{\partial\Omega}: DM_X(\Omega) \times BV_X^c(\Omega) \rightarrow \mathbb{R}$$
and then provide its representation by an $L^\infty$ function.

\begin{theorem}\label{thm:bilinearform}
Assume that $\Omega \subset \mathbb{R}^n$ is an $X$-Lipschitz domain with compact boundary. Then, there exists a bilinear map $\langle \mathbf{z}, u \rangle_{\partial\Omega}: DM_X(\Omega) \times BV_X^c(\Omega) \rightarrow \mathbb{R}$ such that
$$ \langle \mathbf{z}, u \rangle_{\partial\Omega} = \int_{\partial\Omega} Tu \, \mathbf{z} \cdot \nu_\Omega \, d|\partial\Omega|_X \qquad \mbox{ if } \mathbf{z} \in C^1(\overline{\Omega}; \mathbb{R}^m),$$
where $\nu_\Omega$ denotes the horizontal normal, and
$$ |\langle \mathbf{z}, u \rangle_{\partial\Omega}| \leq C(\Omega) \| \mathbf{z} \|_{L^\infty(\Omega; \mathbb{R}^m)} \cdot \| Tu \|_{L^1(\partial\Omega, |\partial\Omega|_X)}.$$
If the domain is $X$-regular, then the previous property holds with constant equal to one:
$$ |\langle \mathbf{z}, u \rangle_{\partial\Omega}| \leq \| \mathbf{z} \|_{L^\infty(\Omega; \mathbb{R}^m)} \cdot \| Tu \|_{L^1(\partial\Omega, |\partial\Omega|_X)}.$$
\end{theorem}

\begin{proof}
In order for the first property to hold, for all $\mathbf{z} \in DM_X(\Omega)$ and $u \in BV_X^c(\Omega) \cap W^{1,1}_X(\Omega)$ we set
$$ \langle \mathbf{z}, u \rangle_{\partial\Omega} = \int_\Omega u \, \mbox{div}_X(\mathbf{z}) \, d\mathcal{L}^n + \int_\Omega \mathbf{z} \cdot Xu \, d\mathcal{L}^n.$$
By Theorem \ref{thm:vittonetraces} the first property holds. Moreover, this map is bilinear. For general $u$, if $Xu$ is merely a measure, the formula above is not well-defined; we will extend it by approximating general $u \in BV_X^c(\Omega)$ using smooth functions. To this end, we notice that if $u, v \in BV_X^c(\Omega) \cap W^{1,1}_X(\Omega)$ have the same trace, then
$$ \langle \mathbf{z}, u \rangle_{\partial\Omega} = \langle \mathbf{z}, v \rangle_{\partial\Omega}.$$
To this end, consider a sequence $g_n \in C_0^\infty(\Omega)$ approximating $u - v$ as in Proposition \ref{prop:approximationpreservingtraces}. We obtain
$$ \langle \mathbf{z}, u - v \rangle_{\partial\Omega} = \int_\Omega (u - v) \, \mbox{div}_X(\mathbf{z}) \, d\mathcal{L}^n + \int_\Omega \mathbf{z} \cdot X(u - v) \, d\mathcal{L}^n =$$
$$ = \lim_{j \rightarrow \infty} \bigg( \int_\Omega g_j \, \mbox{div}_X(\mathbf{z}) \, d\mathcal{L}^n + \int_\Omega \mathbf{z} \cdot X g_j \, d\mathcal{L}^n\bigg) = \lim_{j \rightarrow \infty} \int_{\partial\Omega} 0 \, d|\partial\Omega|_X = 0.$$
Since by Theorem \ref{thm:vittoneextension} for every $u \in BV_X(\Omega)$ there exists a smooth function with the same trace, we define $\langle \mathbf{z}, u \rangle_{\partial\Omega}$ for all $u \in BV_X^c(\Omega)$ by
$$ \langle \mathbf{z}, u \rangle_{\partial\Omega} = \langle \mathbf{z}, v \rangle_{\partial\Omega},$$
where $v$ is any function in $BV_X(\Omega) \cap W^{1,1}_X(\Omega)$ with the same trace as $u$. In view of the calculation above, this uniquely defines $\langle \mathbf{z}, u \rangle_{\partial\Omega}$ for any $u \in BV_X^c(\Omega)$.

Now, we have to prove the second property. Let us take a sequence $u_j \in BV_X^c(\Omega) \cap C^\infty(\Omega)$ which converges to $u$ as in Proposition \ref{prop:approximationpreservingtraces}. Then, we get that
$$ |\langle \mathbf{z}, u \rangle_{\partial\Omega}| = |\langle \mathbf{z}, u_j \rangle_{\partial\Omega}| = |\int_\Omega u_j \, \mbox{div}_X(\mathbf{z}) \, d\mathcal{L}^n + \int_\Omega \mathbf{z} \cdot Xu_j \, d\mathcal{L}^n| \leq $$
$$ \leq | \int_\Omega u_j \, \mbox{div}_X(\mathbf{z}) \, d\mathcal{L}^n | + \| \mathbf{z} \|_{L^\infty(\Omega; \mathbb{R}^m)} \int_\Omega |Xu_j| \, d\mathcal{L}^n.$$
We pass to the limit with $j \rightarrow +\infty$ and obtain
$$ |\langle \mathbf{z}, u \rangle_{\partial\Omega}| \leq | \int_\Omega u \, \mbox{div}_X(\mathbf{z}) \, d\mathcal{L}^n | + \| \mathbf{z} \|_{L^\infty(\Omega; \mathbb{R}^m)} \int_\Omega |Xu|. $$
Fix $\varepsilon > 0$. If the domain is $X$-regular, we may take $u$ to be as in Theorem \ref{thm:vittoneextension}, so that
$$ \int_\Omega |Xu| \leq (1 + \varepsilon) \| Tu \|_{L^1(\partial\Omega, |\partial\Omega|_X)}$$
and $u$ is supported in $\Omega \backslash \Omega_\varepsilon$, where
$$ \Omega_\varepsilon = \{ x \in \Omega: \mbox{ dist}(x, \partial\Omega) > \varepsilon \}.$$
We insert it in the estimate above and obtain
$$ |\langle \mathbf{z}, u \rangle_{\partial\Omega}| \leq | \int_{\Omega \backslash \Omega_\varepsilon} \| u \|_{L^\infty(\Omega)} \, \mbox{div}_X(\mathbf{z}) \, d\mathcal{L}^n | + (1 + \varepsilon) \| \mathbf{z} \|_{L^\infty(\Omega; \mathbb{R}^m)} \| Tu \|_{L^1(\partial\Omega, |\partial\Omega|_X)}. $$
As $\varepsilon$ was arbitrary, we obtain
$$ |\langle \mathbf{z}, u \rangle_{\partial\Omega}| \leq \| \mathbf{z} \|_{L^\infty(\Omega; \mathbb{R}^m)} \| Tu \|_{L^1(\partial\Omega, |\partial\Omega|_X)}. $$
If the domain $\Omega$ is only $X$-Lipschitz, then the approximating sequence instead satisfies
$$ \int_\Omega |Xu| \leq C(\Omega) \| Tu \|_{L^1(\partial\Omega, |\partial\Omega|_X)}$$
so we obtain the final estimate
$$ |\langle \mathbf{z}, u \rangle_{\partial\Omega}| \leq C(\Omega) \| \mathbf{z} \|_{L^\infty(\Omega; \mathbb{R}^m)} \| Tu \|_{L^1(\partial\Omega, |\partial\Omega|_X)}. $$
\end{proof}

Now, we provide a representation of the bilinear map $\langle \mathbf{z}, u \rangle_{\partial\Omega}$ by a linear operator into $L^\infty(\partial\Omega, |\partial\Omega|_X)$.

\begin{theorem}\label{thm:definitionofweaktrace}
Assume that $\Omega \subset \mathbb{R}^n$ is an $X$-Lipschitz domain with compact boundary. Then, there exists a linear operator $\gamma_X: DM_X(\Omega) \rightarrow L^\infty(\partial\Omega, |\partial\Omega|_X)$ such that
$$ \| \gamma_X(\mathbf{z}) \|_{L^\infty(\partial\Omega, |\partial\Omega|_X)} \leq C(\Omega) \| \mathbf{z} \|_{L^\infty(\Omega)},$$
$$ \langle \mathbf{z}, u \rangle_{\partial\Omega} = \int_{\partial\Omega} Tu \, \gamma_X(\mathbf{z}) \, d|\partial\Omega|_X \qquad \mbox{ for all } u \in BV_X^c(\Omega)$$
and
$$ \gamma_X(\mathbf{z}) (x) = \mathbf{z} \cdot \nu_\Omega \qquad \mbox{ if } \mathbf{z} \in C^1(\overline{\Omega}, \mathbb{R}^m).$$
If the domain is $X$-regular, then the constant $C(\Omega)$ equals one. The function $\gamma_X(\mathbf{z})$ is a weakly defined normal trace of $\mathbf{z}$; for this reason, we will denote it by $[\mathbf{z} \cdot \nu]_X$.
\end{theorem}

\begin{proof}
Given $\mathbf{z} \in DM_X(\Omega)$, consider the linear functional $G: L^\infty(\partial\Omega, |\partial\Omega|_X) \rightarrow \mathbb{R}$ defined by the formula
$$ G(f) = \langle \mathbf{z}, u \rangle_{\partial\Omega},$$
where $u \in BV_X^c(\Omega)$ is such that $Tu = f$. By the previous Theorem, we have
$$ |G(f)| = |\langle \mathbf{z}, u \rangle_{\partial\Omega}| \leq C(\Omega) \| \mathbf{z} \|_{L^\infty(\Omega; \mathbb{R}^m)} \cdot \| f \|_{L^1(\partial\Omega, |\partial\Omega|_X)}.$$
As $G$ is a continuous functional on $L^1(\partial\Omega, |\partial\Omega|_X)$, there exists a function $\gamma_X(\mathbf{z}) \in L^\infty(\partial\Omega, |\partial\Omega|_X)$ with norm $C(\Omega) \| \mathbf{z} \|_{L^\infty(\Omega; \mathbb{R}^m)}$ such that
$$ G(u) = \int_{\partial\Omega} f \, \gamma_X(\mathbf{z}) \, d|\partial\Omega|_X.$$
If the domain was $X$-regular, then in the previous Theorem the constant in the estimate on $\langle \mathbf{z}, u \rangle_{\partial_\Omega}$ equals one.
\end{proof}

Now, we introduce the second pairing $(\mathbf{z}, Xu)$. From now on, we consider only Carnot groups; the reason for this is that left-invariant vector fields on Carnot groups satisfy $X^* = -X$, see \cite[Lemma 1.30]{VittPhD}. We will consider the following possible conditions: \\
(a) $u \in BV_X(\Omega) \cap L^\infty(\Omega), \mathbf{z} \in DM_X(\Omega), \mbox{div}_X(\mathbf{z}) \in L^1(\Omega; \mathbb{R}^m);$ \\
(b) $u \in BV_X(\Omega) \cap L^q(\Omega), \mathbf{z} \in DM_X(\Omega), \mbox{div}_X(\mathbf{z}) \in L^p(\Omega; \mathbb{R}^m), 1 < p \leq n, \frac{1}{p} + \frac{1}{q} = 1$; \\
(c) $u \in BV_X^c(\Omega), \mathbf{z} \in DM_X(\Omega).$

\begin{definition}
Suppose that we work with a Carnot group. Let $\mathbf{z}, u$ be such that one of the conditions (a),(b) holds for all open sets $A \subset \subset \Omega$. We define a linear functional $(\mathbf{z}, Xu): C_c^\infty(\Omega) \rightarrow \mathbb{R}$ as
$$ \langle (\mathbf{z}, Xu), \varphi \rangle = - \int_\Omega u \, \varphi \, \mbox{div}_X(\mathbf{z}) \, d\mathcal{L}^n - \int_\Omega u \, \mathbf{z} \cdot X\varphi \, d\mathcal{L}^n.$$
If $\mathbf{z}, u$ are such that (c) holds for all open sets $A \subset \subset \Omega$, we set
$$ \langle (\mathbf{z}, Xu), \varphi \rangle = - \int_\Omega u \, \varphi \, d\mbox{div}_X(\mathbf{z}) - \int_\Omega u \, \mathbf{z} \cdot X\varphi \, d\mathcal{L}^n.$$
\end{definition}

\begin{proposition}\label{prop:boundonanzelottipairing}
Under the assumptions above, the pairing $(\mathbf{z}, Xu)$ is a measure on $\Omega$. Moreover, $(\mathbf{z}, Xu) \ll |Xu|$ and we have that
$$ |(\mathbf{z}, Xu)| \leq \| \mathbf{z} \|_\infty |Xu|$$
as measures on $\Omega$.
\end{proposition}

\begin{proof}
We will prove the Proposition under assumption (a) or (b); the other case is similar. Suppose first that $u \in C^\infty(\Omega)$. Consider the vector field $u \varphi \mathbf{z} \in L^\infty(\Omega; \mathbb{R}^m)$ with compact support in $\Omega$. Then, we have that
\begin{equation*}
0 = \int_\Omega \mbox{div}_X(u \, \varphi \, \mathbf{z}) \, d\mathcal{L}^n
= - \int_\Omega \sum_{j=1}^m X_j^*(u \, \varphi \, \mathbf{z}) d\mathcal{L}^n = - \int_\Omega \sum_{j=1}^m u \, \varphi \, X_j^*(\mathbf{z}_j) d\mathcal{L}^n + \end{equation*}
\begin{equation*}
- \int_\Omega \sum_{j=1}^m \varphi \, \mathbf{z}_j \, X_j^*(u) d\mathcal{L}^n - \int_\Omega \sum_{j=1}^m u \, \mathbf{z}_j \, X_j^*(\varphi) d\mathcal{L}^n = \int_\Omega u \, \varphi \, \mbox{div}_X(\mathbf{z}) d\mathcal{L}^n +
\end{equation*}
\begin{equation*}
- \int_\Omega \varphi \, \mathbf{z} \cdot X^* u \, d\mathcal{L}^n - \int_\Omega u \, \mathbf{z} \cdot X^* \varphi \, d\mathcal{L}^n.
\end{equation*}
As we work on a Carnot group, $X^* = -X$ (see \cite[Lemma 1.30]{VittPhD}); hence
$$ 0 = \int_\Omega u \, \varphi \, \mbox{div}_X(\mathbf{z}) d\mathcal{L}^n + \int_\Omega \varphi \, \mathbf{z} \cdot X u \, d\mathcal{L}^n + \int_\Omega u \, \mathbf{z} \cdot X \varphi \, d\mathcal{L}^n.$$
We rewrite this as
$$ \bigg| \langle (\mathbf{z}, Xu), \varphi \rangle \bigg| = \bigg| - \int_\Omega u \, \varphi \, \mbox{div}_X(\mathbf{z}) d\mathcal{L}^n - \int_\Omega u \, \mathbf{z} \cdot X \varphi \, d\mathcal{L}^n \bigg| = $$
$$ = \bigg| \int_\Omega \varphi \, \mathbf{z} \cdot X u \, d\mathcal{L}^n \bigg| \leq \| \varphi \|_\infty \| \mathbf{z} \|_\infty \int_\Omega |Xu| d\mathcal{L}^n.$$
Now, we drop the assumption of smoothness of $u$: let $u \in BV_X(\Omega)$. Take a sequence $u_j \in C^\infty(\Omega) \cap BV_X(\Omega)$ which converges $X$-strictly to $u$. Then we have
$$ \bigg| \langle (\mathbf{z}, Xu_j), \varphi \rangle \bigg| \leq \| \varphi \|_\infty \| \mathbf{z} \|_\infty \int_\Omega |Xu_j| d\mathcal{L}^n $$
and taking the limit we get 
$$ \bigg| \langle (\mathbf{z}, Xu), \varphi \rangle \bigg| \leq \| \varphi \|_\infty \| \mathbf{z} \|_\infty \int_\Omega |Xu|.$$
Hence, $(\mathbf{z}, Xu)$ is a measure on $\Omega$ which satisfies $|(\mathbf{z}, Xu)| \leq \| \mathbf{z} \|_\infty |Xu|$ as measures. In particular, it is absolutely continuous with respect to $|Xu|$.
\end{proof}

\begin{proposition}
Assume that $u, \mathbf{z}$ satisfy one of the assumptions (a)-(c). Let $u_j \in C^\infty(\Omega) \cap BV_X(\Omega)$ converge to $u$ as in Proposition \ref{prop:approximationpreservingtraces}. Then we have
$$ \int_\Omega (\mathbf{z}, Xu_j) \rightarrow \int_\Omega (\mathbf{z}, Xu).$$
\end{proposition}

\begin{proof}
Fix $\varepsilon > 0$. Then choose an open set $A \subset \subset \Omega$ such that 
$$\int_{\Omega \backslash A} |Xu| < \varepsilon.$$
Let $g \in C_c^\infty(\Omega)$ be such that $0 \leq g \leq 1$ in $\Omega$ and $g \equiv 1$ in $A$. We write $1 = g + (1 - g)$ and estimate
$$ \bigg| \int_\Omega (\mathbf{z}, Xu_j) - \int_\Omega (\mathbf{z}, Xu) \bigg| \leq \bigg| \langle (\mathbf{z}, Xu_j), g \rangle -  \langle (\mathbf{z}, Xu), g \rangle  \bigg| + $$
$$ + \int_\Omega |(\mathbf{z}, Xu_j)| (1-g) + \int_\Omega |(\mathbf{z}, Xu)| (1-g).$$
But for any fixed $g \in C_c^\infty(\Omega)$ we have $\langle (\mathbf{z}, Xu_j), g \rangle \rightarrow \langle (\mathbf{z}, Xu), g \rangle$. Moreover, we have 
$$\int_\Omega (1-g) |(\mathbf{z}, Xu)| \leq \int_{\Omega \backslash A} |(\mathbf{z}, Xu)| \leq \| \mathbf{z} \|_\infty \int_{\Omega \backslash A} |Xu| < \varepsilon \| \mathbf{z} \|_\infty$$
and similarly
$$\limsup_{j \rightarrow \infty} \int_\Omega (1-g) |(\mathbf{z}, Xu_n)| \leq \limsup_{j \rightarrow \infty} \| \mathbf{z} \|_\infty \int_{\Omega \backslash A} |Xu_n| \leq \varepsilon \| \mathbf{z} \|_\infty ,$$
so the right hand side may be arbitrarily close to zero.
\end{proof}

We conclude  by introducing the Green's formula, which relates the measure $(\mathbf{z}, Xu)$ with the weak normal trace.

\begin{theorem}\label{thm:greenformulacarnot}
Let $\Omega$ be a bounded open set with $X-$Lipschitz boundary. Let $u, \mathbf{z}$ satisfy one of the conditions (a),(b). Then we have
\begin{equation}
\int_\Omega u \, \mbox{div}_X(\mathbf{z}) d\mathcal{L}^n + \int_\Omega (\mathbf{z}, Xu) = \int_{\partial\Omega} [\mathbf{z} \cdot \nu]_X \, u \, d|\partial\Omega|_X.
\end{equation}
If $u, \mathbf{z}$ satisfy condition (c), we have
\begin{equation}
\int_\Omega u \, d\mbox{div}_X(\mathbf{z}) + \int_\Omega (\mathbf{z}, Xu) = \int_{\partial\Omega} [\mathbf{z} \cdot \nu]_X \, u \, d|\partial\Omega|_X.
\end{equation}
\end{theorem}

\begin{proof}
We will prove the Proposition under assumption (a) or (b); the other case is similar. Take a sequence $u_j \in C^\infty(\Omega) \cap BV_X(\Omega)$ that converges to $u$ as in Proposition \ref{prop:approximationpreservingtraces}. Then, by the previous Proposition and Theorem \ref{thm:vittonetraces}, we have
$$ \int_\Omega u \, \mbox{div}_X(\mathbf{z}) d\mathcal{L}^n + \int_\Omega (\mathbf{z}, Xu) = \lim_{j \rightarrow \infty} \bigg( \int_\Omega u_j \, \mbox{div}_X(\mathbf{z}) d\mathcal{L}^n + \int_\Omega (\mathbf{z}, Xu_j) \bigg) =$$
$$ = \lim_{j \rightarrow \infty} \int_{\partial\Omega} [\mathbf{z} \cdot \nu]_X \, u_j \, d|\partial\Omega|_X = \int_{\partial\Omega} [\mathbf{z} \cdot \nu]_X \, u \, d|\partial\Omega|_X. $$
\end{proof}

\subsection{Least gradient problem on a Carnot group} Now, we introduce three notions of solutions to the least gradient problem on a Carnot group: as minimizers of a variational functional with a rigid trace constraint, as minimizers to a relaxed variational functional and as solutions to the (local) 1-Laplace problem. We will later show that these three approaches are almost equivalent.

\begin{definition}\label{dfn:leastgradientcarnot}
We say that $u \in BV_X(\Omega)$ is a function of least gradient, if for any $v \in BV_X(\Omega)$ such that $Tv = 0$ we have
$$ \int_\Omega |Xu| \leq \int_\Omega |X(u+v)|.$$
\end{definition}

\begin{definition}\label{dfn:functionalcarnot}
Consider a functional $F_h: L^1(\Omega) \rightarrow \mathbb{R}$ defined by the formula
$$F_h(u) = \int_\Omega |Xu| + \int_{\partial\Omega} |h-u| \, d|\partial\Omega|_X$$
with $F_h(u) = +\infty$ if $u \in L^1(\Omega) \backslash BV_X(\Omega)$. We say that $u$ is a solution to the least gradient problem on $\Omega$ with boundary data $h \in L^1(\partial\Omega)$, if $u \in \mbox{argmin } F_h$.
\end{definition}

These definitions are analogues of Euclidean definitions on Carnot groups, see for instance \cite{MRL}. The main difference is that if instead of adopting Definition \ref{dfn:functionalcarnot} we require that $u$ is a function of least gradient such that $Tu=f$, then even in the Euclidean case we may lose existence of solutions for any $h \in L^1(\Omega)$. Moreover, as we will see, Definition \ref{dfn:functionalcarnot} has a direct relation to the Euler-Lagrange formulation of the problem presented below.

\begin{definition}\label{dfn:1laplacecarnot}
We say that $u \in BV_X(\Omega)$ is a solution of the 1-Laplace equation with Dirichlet boundary data on a Carnot group for boundary data $h \in L^1(\partial\Omega, |\partial\Omega|_X)$, if there exists a vector field $\mathbf{z} \in L^\infty(\Omega; \mathbb{R}^m)$ such that $\| \mathbf{z} \|_{L^\infty(\Omega; \mathbb{R}^m)} \leq 1$ and the following conditions hold:
\begin{equation}
\mbox{div}_X(\mathbf{z}) = 0 \qquad \mbox{as distributions in }\Omega ;
\end{equation}
\begin{equation}
(\mathbf{z}, Xu) = |Xu| \qquad \mbox{as measures in }\Omega;
\end{equation}
\begin{equation}
[\mathbf{z} \cdot \nu]_X \in \mbox{sign}(h-u) \qquad |\partial\Omega|_X-\mbox{a.e. on }\partial\Omega.
\end{equation}
\end{definition}

At this time, we will not prove existence of solutions to the Dirichlet problem for the 1-Laplace operator; we will do this later in Theorem \ref{thm:connectionwithlocal} using an approximation by nonlocal problems with range of the interaction going to zero.

\begin{proposition}
Assume that $\Omega$ is $X$-regular. Given $u \in BV_X(\Omega)$, consider the following conditions: \\
(i) $F_h(u) \leq F_h(v)$ for all $v \in  BV_X(\Omega)$; \\
(ii) $u$ is a solution of the 1-Laplace equation; \\
(iii) $u$ is a function of least gradient. \\
Then (ii) implies (i). If there exists a solution $\overline{u} \in BV_X(\Omega)$ in the sense of Definition \ref{dfn:1laplacecarnot}, then also (i) implies (ii). Furthermore, if $Tu = h$, then (i) and (iii) are equivalent.
\end{proposition}

\begin{proof}
(ii) $\Rightarrow$ (i). Given $w \in BV(\Omega)$, we apply the Green's formula (Theorem \ref{thm:greenformulacarnot}) to see that
$$ 0 = \int_\Omega (w - u) \, \mbox{div}_X(\mathbf{z}) \, d\mathcal{L}^n = - \int_\Omega (\mathbf{z}, Xw) + \int_\Omega (\mathbf{z}, Xu) + \int_{\partial\Omega} [\mathbf{z} \cdot \nu]_X \, (w-u) \, d|\partial\Omega|_X.$$
We use the equality above in the definition of $F_h$ to see that
$$ F_h(u) = \int_\Omega |Xu| + \int_{\partial\Omega} |h-u| \, d|\partial\Omega|_X = \int_\Omega (\mathbf{z}, Xu) + \int_{\partial\Omega} [\mathbf{z} \cdot \nu]_X (h-u) \, d|\partial\Omega|_X = $$
$$ =\int_\Omega (\mathbf{z}, Xw) + \int_{\partial\Omega} [\mathbf{z} \cdot \nu]_X (h-w) \, d|\partial\Omega|_X \leq \int_\Omega |Xw| + \int_{\partial\Omega} |h-w| \, d|\partial\Omega|_X = F_h(w),$$
where in the last estimate we used $X$-regularity of $\Omega$, as we estimate $|[\mathbf{z} \cdot \nu]_X| \leq \| \mathbf{z} \|_\infty \leq 1$; by Theorem \ref{thm:bilinearform} this estimate without an additional constant holds only for $X$-regular domains. \\
\\
(i) $\Rightarrow$ (ii). We assume that there exists a solution $\overline{u} \in BV_X(\Omega)$ in the sense of Definition \ref{dfn:1laplacecarnot}; in particular, take the vector field $\mathbf{z} \in L^\infty(\mathbb{R}^n; \mathbb{R}^m)$ subordinate to $\overline{u}$. Using Green's formula (Theorem \ref{thm:greenformulacarnot}), we have that 
$$ 0 = \int_\Omega (\overline{u} - u) \, \mbox{div}_X(\mathbf{z}) \, d\mathcal{L}^n = - \int_\Omega (\mathbf{z}, X\overline{u}) + \int_\Omega (\mathbf{z}, Xu) + \int_{\partial\Omega} [\mathbf{z} \cdot \nu]_X \, (\overline{u}-u) \, d|\partial\Omega|_X.$$
By the minimality of $u$ and the above equality, we have
$$ F_h(u) \leq F_h(\overline{u}) = \int_\Omega |X\overline{u}| + \int_{\partial\Omega} |h-\overline{u}| \, d|\partial\Omega|_X = \int_\Omega (\mathbf{z}, X\overline{u}) + $$
$$ + \int_{\partial\Omega} [\mathbf{z} \cdot \nu]_X (h-\overline{u}) \, d|\partial\Omega|_X = \int_\Omega (\mathbf{z}, Xu) + \int_{\partial\Omega} [\mathbf{z} \cdot \nu]_X (h-u) \, d|\partial\Omega|_X \leq $$
$$ \leq \int_\Omega |Xu| + \int_{\partial\Omega} |h-u| \, d|\partial\Omega|_X  = F_h(u),$$
where in the inequality we used $X$-regularity of $\Omega$ in the same way as in the converse implication. In particular, this inequality is in fact an equality, so
$$ \int_\Omega \bigg( |Xu| - (\mathbf{z}, Xu) \bigg) + \int_{\partial\Omega} \bigg(|h-u| - [\mathbf{z} \cdot \nu]_X (h-u) \bigg) \, d|\partial\Omega|_X = 0.$$
Since $\Omega$ is $X$-regular, both integrands are nonnegative, hence $(\mathbf{z}, Xu) = |Xu|$ as measures and $|h-u| = [\mathbf{z} \cdot \nu]_X (h-u)$ almost everywhere with respect to $|\partial\Omega|_X$, so $u$ satisfies Definition \ref{dfn:1laplacecarnot}. \\
\\
(i) $\Rightarrow$ (iii). Since $Tu = h$, by minimality of $u$ for $v \in BV_X(\Omega)$ with $Tv = 0$ we have
$$ \int_\Omega |Xu| = F_h(u) \leq F_h(u+v) = \int_\Omega |X(u+v)|.$$
(iii) $\Rightarrow$ (i). Given $w \in BV_X(\Omega)$, we have to prove that $F_h(u) \leq F_h(w)$. As $\Omega$ is $X$-regular, we fix $\varepsilon > 0$ and use Theorem \ref{thm:vittoneextension} to find a function $\widetilde{w} \in C^\infty(\Omega) \cap BV_X(\Omega)$ such that
$$ \widetilde{w} = w - h \quad \mbox{on } \partial\Omega;$$
$$ \int_\Omega |X\widetilde{w}| \leq (1 + \varepsilon) \int_{\partial\Omega} |w-h| \, d|\partial\Omega|_X.$$
Consider the function $w - \widetilde{w}$ with trace $h$. Since $u$ is a function of least gradient with trace $h$, we have
$$ F_h(u) = \int_\Omega |Xu| \leq \int_\Omega |X(w - \widetilde{w})| \leq \int_\Omega |Xw| + \int_\Omega |X\widetilde{w}| \leq $$
$$ \leq \int_\Omega |Xw| + (1 + \varepsilon) \int_{\partial\Omega} |w-h| \, d|\partial\Omega|_X = F_h(w) + \varepsilon \int_{\partial\Omega} |w-h| \, d|\partial\Omega|_X.$$
As $\varepsilon$ is arbitrary, we see that $u \in \mbox{argmin} \, F_h$.
\end{proof}

\section{Convergence of the nonlocal least gradient functions}\label{sec:4nonlocalconvergence}

In this Section, we will prove convergence of minimizers of nonlocal least gradient problems in a Carnot group to a solution of the local least gradient problem. For this, we need a few preliminary results. We start with an explicit calculation of the limit of a certain difference quotient, which we will use to prove Theorem \ref{thm:amrtincarnotgroups} - a result on the structure of the derivative of a weak limit of a sequence of functions with uniformly bounded nonlocal gradients. Then, we will use Theorem \ref{thm:amrtincarnotgroups} to prove the aforementioned convergence of minimisers.

\subsection{Structure of the limit}

One of the steps in the proof of Theorem \ref{thm:amrtincarnotgroups} will be calculating on a Carnot group the following object:
$$ \frac{d}{d\varepsilon} \bigg|_{\varepsilon = 0} \varphi(x \delta_\varepsilon(z^{-1})).$$
In the Euclidean case, this reduces to
$$ \frac{d}{d\varepsilon} \bigg|_{\varepsilon = 0} \varphi(x - \varepsilon z) = - z \cdot \nabla \varphi(x).$$
We want to perform a similar calculation in a Carnot group $\mathbb{G}$. Firstly, we present two examples: in the Heisenberg group $\mathbb{H}^1$ and in the Engel group $\mathbb{E}^4$, which will serve as a toy problem. Then, we prove the appropriate result in full generality.

\begin{example}\label{ex:heisenbergderivative}
For the definition of the operations in the Heisenberg group $\mathbb{H}^1$ and the horizontal vector fields defining it, see Example \ref{ex:heisenberggroup}. We use the group law in the Heisenberg group to explicitly compute the derivative: we write $x= (x_1, x_2, x_3)$, $z= (z_1, z_2, z_3)$ and calculate
$$ \frac{d}{d\varepsilon}\bigg|_{\varepsilon = 0} \, \varphi(x \delta_\varepsilon(z^{-1})) = \frac{d}{d\varepsilon}\bigg|_{\varepsilon = 0} \, \varphi((x_1, x_2, x_3) \circ (-\varepsilon z_1, - \varepsilon z_2, - \varepsilon^2 z_3)) = $$
$$ = \frac{d}{d\varepsilon}\bigg|_{\varepsilon = 0} \, \varphi((x_1 - \varepsilon z_1, x_2 - \varepsilon z_2, x_3 - \varepsilon^2 z_3 + \frac12 \varepsilon (- x_1 z_2 + x_2 z_1))) = $$
$$ = -z_1 \cdot \partial_1 \varphi(x) - z_2 \cdot \partial_2 \varphi(x) - \frac12 x_1 z_2 \cdot \partial_3 \varphi(x) + \frac12 x_2 z_1 \cdot \partial_3 \varphi(x) = $$
$$ = - z_1 (\partial_1 - \frac12 x_2 \cdot \partial_3) \varphi(x) - z_2 (\partial_2 + \frac12 x_1 \cdot \partial_3) \varphi(x) = - z_1 \cdot X_1 \varphi(x) - z_2 \cdot X_2 \varphi(x) = - \langle z, X\varphi \rangle.$$
\end{example}

\begin{example}\label{ex:engelderivative}

The Engel group $\mathbb{E}^4$ is the space $\mathbb{R}^4$ endowed with the group operation
$$ (x_1, x_2, x_3, x_4) \circ (y_1, y_2, y_3, y_4) = (x_1 + y_1, x_2 + y_2, x_3 + y_3 + \frac12 (x_1 y_2 - x_2 y_1),$$
$$x_4 + y_4 + \frac12 (x_1 y_3 - x_3 y_1) + \frac12 (x_2 y_3 - x_3 y_2) + \frac{1}{12}(x_1 + x_2 - y_1 - y_2)(x_1 y_2 - x_2 y_1)),$$
dilations
$$ \delta_\lambda((x_1, x_2, x_3, x_4)) = (\lambda x_1, \lambda x_2, \lambda^2 x_3, \lambda^3 x_4)$$
and with horizontal vector fields
$$ X_1 = \partial_1 - \frac{x_2}{2} \partial_3 - (\frac{x_3}{2} + \frac{x_2^2}{12} + \frac{x_1 x_2}{12}) \partial_4, \quad X_2 = \partial_2 + \frac{x_1}{2} \partial_3 - (\frac{x_3}{2} - \frac{x_2^2}{12} - \frac{x_1 x_2}{12}) \partial_4.$$
It is a Carnot group of step 3. In the Engel group, the calculation looks as follows:
$$ \frac{d}{d\varepsilon}\bigg|_{\varepsilon = 0} (x \delta_\varepsilon(z^{-1})) = \frac{d}{d\varepsilon}\bigg|_{\varepsilon = 0} ((x_1, x_2, x_3, x_4) ((- \varepsilon z_1, -\varepsilon z_2, -\varepsilon^2 z_3, -\varepsilon^3 z_4)).$$
Notice that due to the graded structure of $\mathbb{E}^4$ no matter what are the exact coefficients on the last two coordinates, they will disappear in the limit $\varepsilon \rightarrow 0$; the only input remaining after taking the limit are terms linear in $\varepsilon$, i.e. the ones coming from the horizontal directions $z_1, z_2$. We perform the multiplication and see that
$$ \frac{d}{d\varepsilon}\bigg|_{\varepsilon = 0} (x \delta_\varepsilon(z^{-1})) = $$
$$ = \frac{d}{d\varepsilon}\bigg|_{\varepsilon = 0} ((x_1 - \varepsilon z_1 , x_2 -\varepsilon z_2, x_3 + \frac12 \varepsilon (x_2 z_1 - x_1 z_2) + O(\varepsilon^2), $$
$$ x_4 + \frac12 \varepsilon (x_3 z_1 + x_3 z_2) + \frac{1}{12} \varepsilon (x_1 x_2 z_1 + x_2^2 z_1 - x_1 x_2 z_2 - x_1^2 z_2) + O(\varepsilon^2)).$$
We pass to the limit and see that
$$ \frac{d}{d\varepsilon}\bigg|_{\varepsilon = 0} (x \delta_\varepsilon(z^{-1})) = - z_1 (\partial_1 - \frac{x_2}{2} \partial_3 - (\frac{x_3}{2} + \frac{x_2^2}{12} + \frac{x_1 x_2}{12}) \partial_4) \varphi(x) + $$
$$ - z_2 (\partial_2 + \frac{x_1}{2} \partial_3 - (\frac{x_3}{2} - \frac{x_2^2}{12} - \frac{x_1 x_2}{12}) \partial_4) \varphi(x) = -z_1 X_1 \varphi(x) - z_2 X_2 \varphi(x) = - \langle z, X\varphi \rangle.$$
\end{example}

In general, the main problem is that given horizontal vector fields $X_1,...,X_m$, we can use the Baker-Campbell-Hausdorff formula to obtain the group law in $\mathbb{G}$, but the result may be quite complicated. However, as we saw in Example \ref{ex:engelderivative}, even though the group operation is such that the result of $x \delta_\varepsilon(z^{-1})$ becomes tiresome to compute, the graded structure of the group forces some terms (including the coordinates $z_i$ corresponding to nonhorizontal coordinates) to disappear in the limit. In fact, due to properties of exponential coordinates listed in Proposition \ref{prop:propertiesexponentialcoordinates} we know enough to perform the calculation in the general case. In particular, if $x=(x_1,...,x_n)$ is a representation of $x$ in exponential coordinates of a Carnot group $\mathbb{G}$, then in these coordinates the group multiplication is represented as a polynomial function and the inverse is simply $x^{-1} = -x$. We will use the above facts to prove the following Lemma.

\begin{lemma}\label{lem:horizontalgradient}
Suppose that $\mathbb{G}$ is a Carnot group represented in exponential coordinates. Suppose that $i=1,...,m$ are the horizontal directions. Then
$$ \frac{d}{d\varepsilon} \bigg|_{\varepsilon =  0} \varphi(x \delta_\varepsilon(z^{-1})) = - \sum_{j=1}^m z_i \, X_i \varphi(x) = - \langle z, X \varphi \rangle.$$
\end{lemma}

\begin{proof}
Recall that by Proposition \ref{prop:propertiesexponentialcoordinates} the group multiplication is a polynomial function. In coordinates, notice that for $y$ of the form $y = (y_1,...,y_m,0,...,0)$ it has the form
$$ (x_1,...,x_n) \circ (y_1,...,y_m,0,...,0) = (x_1 + y_1, ..., x_m + y_m, $$
$$x_{m+1} + \sum_{j=1}^m W_{m+1}^j(x) y_j + R_{m+1}(x,y), ..., x_{n} + \sum_{j=1}^m W_{n}^j(x) y_j + R_{n}(x,y)),$$
i.e. it is a linear function on the horizontal coordinates (by Proposition \ref{prop:propertiesexponentialcoordinates}) and on all the other coordinates we decide to write the multiplication in the form as above, where $W_{k}^j(x)$ are polynomials of $x_1,...,x_n$ and $R_k(x,y)$ are polynomials of $x_1,...,x_n,y_1,...,y_m$ involving only quadratic (or higher) terms in $y_1,...,y_m$. In particular, we have
$$ (x_1,...,x_n) \circ (-\varepsilon z_1,...,-\varepsilon z_m,0,...,0) = (x_1 - \varepsilon z_1, ..., x_m - \varepsilon z_m, $$
$$x_{m+1} - \varepsilon \sum_{j=1}^m W_{m+1}^j(x) z_j + O(\varepsilon^2), ..., x_{n} - \varepsilon \sum_{j=1}^m W_{n}^j(x) z_j + O(\varepsilon^2)).$$
Now, we know enough structure of the group operation to compute the desired derivative. We have
$$ \frac{d}{d\varepsilon} \bigg|_{\varepsilon = 0} \varphi(x \delta_\varepsilon(z^{-1})) =$$
$$=\frac{d}{d\varepsilon} \bigg|_{\varepsilon = 0} \varphi((x_1,...,x_m,x_{m+1}...,x_n) \delta_\varepsilon((-z_1,...,-z_m,-z_{m+1},...,-z_{n})) = $$
$$=\frac{d}{d\varepsilon} \bigg|_{\varepsilon = 0} \varphi((x_1,...,x_m,x_{m+1}...,x_n) (-\varepsilon z_1,...,- \varepsilon z_m,O(\varepsilon^2),...,O(\varepsilon^2))).$$
In particular, all terms involving $z_{m+1},...,z_n$ will disappear in the limit, so
$$\frac{d}{d\varepsilon} \bigg|_{\varepsilon = 0} \varphi(x \delta_\varepsilon(z^{-1})) = \frac{d}{d\varepsilon} \bigg|_{\varepsilon = 0} \varphi((x_1,...,x_m,x_{m+1}...,x_n) (-\varepsilon z_1,...,- \varepsilon z_m,0,...,0)) =$$
$$ = \frac{d}{d\varepsilon} \bigg|_{\varepsilon = 0} \varphi((x_1 - \varepsilon z_1, ..., x_m - \varepsilon z_m, $$
$$ x_{m+1} - \varepsilon \sum_{j=1}^m W_{m+1}^j(x) z_j + O(\varepsilon^2), ..., x_{n} - \varepsilon \sum_{j=1}^m W_{n}^j(x) z_j + O(\varepsilon^2))) = $$
$$ = - \sum_{j=1}^m z_j (\partial_j + W_{m+1}^j(x) \partial_{m+1} + ... + W_n^j(x) \partial_n) \varphi(x).$$
In particular, the result is a linear function of $z_1,...,z_m$ - all terms involving $z_i$ with a higher index will disappear in the limit. As the group multiplication is a polynomial function, also all terms involving higher powers of $z_i$ will have at least $\varepsilon^2$ in front of them and disappear in the limit. We will see that the differential operator appearing above reduces to the horizontal gradient. To see, let us calculate for $j = 1,...,m$
$$ X_j \varphi(x) = \frac{d}{d\varepsilon} \bigg|_{\varepsilon = 0} \varphi(x \circ \exp(\varepsilon X_j)) = \frac{d}{d\varepsilon} \bigg|_{\varepsilon = 0} \varphi((x_1,...,x_n) \circ (0,...,0,\varepsilon,0,...,0)) =$$
$$ = \frac{d}{d\varepsilon} \bigg|_{\varepsilon = 0} \varphi((x_1,...,x_{j-1},x_j + \varepsilon, x_{j+1},..., x_m, $$
$$ x_{m+1} + \varepsilon W_{m+1}^j(x) + O(\varepsilon^2), ..., x_n + \varepsilon W_{n}^j(x) + O(\varepsilon^2))) = $$
$$ = (\partial_j + W_{m+1}^j(x) \partial_{m+1} + ... + W_{n}^j(x) \partial_{n}) \varphi(x).$$
Hence, we plug it in the calculation above and get
$$ \frac{d}{d\varepsilon} \bigg|_{\varepsilon = 0} \varphi(x \delta_\varepsilon(z^{-1})) = - \sum_{j=1}^m z_j X_j\varphi(x)$$
and the Lemma is proved. We recall that $\langle z, X\varphi \rangle := \sum_{j=1}^m z_j X_j\varphi(x)$; we again stress that this is not the usual scalar product, as $z$ has $n$ coordinates and $X\varphi$ has $m$ coordinates.
\end{proof}

Now, we want to prove a result analogous to the first point of \cite[Theorem 6.11]{AMRT}. We use the notation so that each function is prolonged by zero outside of $D$.

\begin{theorem}\label{thm:amrtincarnotgroups}
Let $\mathbb{G}$ be a Carnot group represented in exponential coordinates. Assume that $f_\varepsilon \rightharpoonup f$ in $L^q(D)$ for $q \geq 1$ and that the sequence $f_\varepsilon$ satisfies
\begin{equation}\label{eq:uniformestimate} 
\int_D \dashint_{U(x,\varepsilon)} |f_\varepsilon(y) - f_\varepsilon(x)|^q \, d\mathcal{L}^n(x) \, d\mathcal{L}^n(y) \leq M \varepsilon^q.
\end{equation}
If $q > 1$, then $f \in W^{1,q}_X(D)$. If $q = 1$, then $f \in BV_X(D)$. Moreover, on a subsequence $($still denoted by $f_\varepsilon)$ we have
\begin{equation}\label{eq:structureofthederivative}
\chi_{U(0,1)}(z) \, \chi_D(x \delta_\varepsilon(z)) \, \frac{f_\varepsilon(x \delta_\varepsilon(z)) - f_\varepsilon(x)}{\varepsilon} \rightharpoonup \chi_{U(0,1)}(z) \, \cdot \, \langle z, Xf \rangle
\end{equation}
weakly as functions in $L^q(D \times \mathbb{R}^n)$ $($if $q > 1)$ or as measures $($if $q = 1)$.
\end{theorem}

\begin{proof}
We start by rewriting the estimate \eqref{eq:uniformestimate} using dilations in the group $\mathbb{G}$, so that we have an estimate on a fixed domain and not on a set changing with $\varepsilon$. To this end, we will utilise the group structure of $\mathbb{G}$, the invariance of the Lebesgue measure, rescaling using the dilation $\delta_\lambda$ and the behaviour of the Lebesgue measure with respect to this scaling.

We rewrite equation \eqref{eq:uniformestimate} as
\begin{equation}\label{eq:uniformestimate2} 
\int_D \int_{U(x,\varepsilon)} \frac{1}{\mathcal{L}^n(U(x,\varepsilon))} \bigg|\frac{f_\varepsilon(y) - f_\varepsilon(x)}{\varepsilon}\bigg|^q \, d\mathcal{L}^n(x) \, d\mathcal{L}^n(y) \leq M.
\end{equation}
Recall that by the Ahlfors regularity of $\mathcal{L}^n$ we have that $\mathcal{L}^n(U(x,\varepsilon)) = C \varepsilon^Q$, where $C = \mathcal{L}^n(U(0,1))$ is the measure of the unit ball and $Q$ is the homogenous dimension of $\mathbb{G}$. We also notice that if $y \in U(x, \varepsilon)$, then $x^{-1} y \in U(0, \varepsilon)$. Furthermore, if we set $z = \delta_{\varepsilon^{-1}}(x^{-1}y)$, then $z \in U(0,1)$. We will use this change of variables, so that $y = x \delta_\varepsilon(z)$ and $d\mathcal{L}^n(y) = \varepsilon^Q d\mathcal{L}^n(z)$. Hence, equation \eqref{eq:uniformestimate2} takes form
\begin{equation}\label{eq:uniformestimate3} 
\int_D \int_{U(0,1)} \frac{1}{C \varepsilon^Q} \chi_D(x \delta_\varepsilon(z)) \bigg|\frac{f_\varepsilon(x \delta_\varepsilon(z)) - f_\varepsilon(x)}{\varepsilon}\bigg|^q \, \varepsilon^Q d\mathcal{L}^n(x) \, d\mathcal{L}^n(z) \leq M.
\end{equation}
The factors $\varepsilon^Q$ cancel out; moving the constant $C$ (which does not depend on $\varepsilon$) into $M$ and representing the integration over $U(0,1)$ as a characteristic function, we obtain that
\begin{equation}\label{eq:uniformestimate4} 
\int_D \int_{\mathbb{R}^n} \chi_{U(0,1)}(z) \,\chi_D(x \delta_\varepsilon(z)) \, \bigg|\frac{f_\varepsilon(x \delta_\varepsilon(z)) - f_\varepsilon(x)}{\varepsilon}\bigg|^q \,  d\mathcal{L}^n(x) \, d\mathcal{L}^n(z) \leq M.
\end{equation}
As the integrand is uniformly bounded in $L^q(D \times \mathbb{R}^n)$, up to a subsequence we have
\begin{equation}
\chi_{U(0,1)}(z) \, \chi_D(x \delta_\varepsilon(z)) \, \frac{f_\varepsilon(x \delta_\varepsilon(z)) - f_\varepsilon(x)}{\varepsilon} \rightharpoonup \chi_{U(0,1)}(z) \, g(x,z)
\end{equation}
weakly in $L^q(D \times \mathbb{R}^n)$ (if $q > 1$) or
\begin{equation}
\chi_{U(0,1)}(z) \, \chi_D(x \delta_\varepsilon(z)) \, \frac{f_\varepsilon(x \delta_\varepsilon(z)) - f_\varepsilon(x)}{\varepsilon} \rightharpoonup \mu(x,z)
\end{equation}
weakly in $\mathcal{M}(D \times \mathbb{R}^n)$ (if $q = 1$). We want to obtain a representation of the function $g$ (or the measure $\mu$) in terms of the partial derivatives $X_i$ of $f$, which would show that these are functions in $L^q(D)$ (if $q > 1$), so that $f \in W^{1,q}_X(D)$, or bounded Radon measures (if $q = 1$), so that $f \in BV_X(D)$. 

We multiply the expression above by a smooth test function with separated variables: let $\varphi \in \mathcal{D}(D)$ and $\psi \in \mathcal{D}(\mathbb{R}^n)$, then for sufficiently small $\varepsilon$ we may move the differential quotient from $f_\varepsilon$ onto $\varphi$ and we have
\begin{equation}\label{eq:discreteintegrationbyparts1} 
\int_D \int_{\mathbb{R}^n} \chi_{U(0,1)}(z) \,\chi_D(x \delta_\varepsilon(z)) \, \frac{f_\varepsilon(x \delta_\varepsilon(z)) - f_\varepsilon(x)}{\varepsilon} \, \varphi(x) \, \psi(z) \, d\mathcal{L}^n(x) \, d\mathcal{L}^n(z) =
\end{equation}
\begin{equation*} 
= \int_{\mbox{supp}(\varphi)} \int_{\mathbb{R}^n} \chi_{U(0,1)}(z) \, \frac{f_\varepsilon(x \delta_\varepsilon(z)) - f_\varepsilon(x)}{\varepsilon} \, \varphi(x) \, \psi(z) \, d\mathcal{L}^n(x) \, d\mathcal{L}^n(z) =
\end{equation*}
\begin{equation*} 
= \int_{\mathbb{R}^n} \chi_{U(0,1)}(z) \, \psi(z) \, \bigg( \int_{D} \frac{f_\varepsilon(x \delta_\varepsilon(z)) - f_\varepsilon(x)}{\varepsilon} \, \varphi(x) \, d\mathcal{L}^n(x) \bigg) \, d\mathcal{L}^n(z) =
\end{equation*}
\begin{equation*} 
= - \int_{\mathbb{R}^n} \chi_{U(0,1)}(z) \, \psi(z) \, \bigg( \int_{D} f_\varepsilon(x) \frac{\varphi(x) - \varphi(x \delta_\varepsilon(z^{-1}))}{\varepsilon} \, d\mathcal{L}^n(x) \bigg) \, d\mathcal{L}^n(z).
\end{equation*}
We want to pass to the limit with $\varepsilon \rightarrow 0$. Firstly, we will investigate the differential quotient involving $\varphi$; once we establish the form of its limit and its continuity, we may pass with $f_\varepsilon$ to the limit using the assumption of weak convergence. We need to calculate the directional derivative of $\varphi$ in the direction $z$. We notice that
$$ \lim_{\varepsilon \rightarrow 0} \frac{\varphi(x) - \varphi(x \delta_\varepsilon(z^{-1}))}{\varepsilon} = - \frac{d}{d\varepsilon}\bigg|_{\varepsilon = 0} \, \varphi(x \delta_\varepsilon(z^{-1}))$$
whenever the limit is defined. By Lemma \ref{lem:horizontalgradient} we have that
$$ \frac{d}{d\varepsilon} \bigg|_{\varepsilon =  0} \varphi(x \delta_\varepsilon(z^{-1})) = - \langle z, X \varphi \rangle := - \sum_{j=1}^m z_i \, X_i \varphi(x).$$
This limit is precisely $z$ multiplied by the horizontal gradient of $\varphi$ (with zero on the non-horizontal coordinates). 

Suppose that $q > 1$. We pass to the limit in \eqref{eq:discreteintegrationbyparts1} to obtain
\begin{equation}
\int_{D \times \mathbb{R}^n} \chi_{U(0,1)}(z) \, g(x,z) \, \varphi(x) \, \psi(z) \, d\mathcal{L}^n(x) \, d\mathcal{L}^n(z) = 
\end{equation}
\begin{equation}
= - \int_{D \times \mathbb{R}^n} \chi_{U(0,1)}(z) \, f(x) \, \langle z, X\varphi \rangle \, \psi(z) \,  d\mathcal{L}^n(x)\,  d\mathcal{L}^n(z).
\end{equation}
By choosing appropriately the functions $\psi \in \mathcal{D}(\mathbb{R}^n)$, we see that
\begin{equation}
\int_{D} g(x,z) \, \varphi(x) \, d\mathcal{L}^n(x) = - \int_{D} f(x) \, \langle z, X\varphi \rangle \, d\mathcal{L}^n(x) \qquad \mbox{for all } z \in U(0,1).
\end{equation}
We recall that the adjoint operator of a left-invariant vector field $X$ is $-X$ (\cite[Lemma 1.30]{VittPhD}). Hence, on the right hand side we may move the horizontal gradient onto $f$ to obtain
\begin{equation}
\int_{D} g(x,z) \, \varphi(x) \, d\mathcal{L}^n(x) = \int_{D} \varphi(x) \, \langle z, Xf \rangle \, d\mathcal{L}^n(x) \qquad \mbox{for all } z \in U(0,1),
\end{equation}
in particular, by choosing $z = s e_i$ for sufficiently small $s$ we see that the components $X_i f$ are functions in $L^q(D)$, so $f \in W^{1,q}(D)$; moreover, $g = \langle z, Xf \rangle$, which proves \eqref{eq:structureofthederivative}.

Now, suppose that $q = 1$. We pass to the limit in \eqref{eq:discreteintegrationbyparts1} to obtain
\begin{equation}\label{eq:discreteintegrationbyparts2}
\int_{D \times \mathbb{R}^n} \varphi(x) \, \psi(z) \, d\mu(x,z) = - \int_{D \times \mathbb{R}^n} \chi_{U(0,1)}(z) \, \psi(z) \, \langle z, X\varphi \rangle \, f(x) d\mathcal{L}^n(x)\,  d\mathcal{L}^n(z).
\end{equation}
Again, we use the fact that the adjoint operator of a left-invariant vector field $X$ is $-X$. Hence, on the right hand side we may move the horizontal gradient onto $f$ to obtain
\begin{equation}\label{eq:discreteintegrationbyparts3}
\int_{D \times \mathbb{R}^n} \varphi(x) \, \psi(z) \, d\mu(x,z) = \int_{D \times \mathbb{R}^n} \chi_{U(0,1)}(z) \, \psi(z) \, \langle z, Xf \rangle \, \varphi(x) d\mathcal{L}^n(x)\,  d\mathcal{L}^n(z).
\end{equation}
By the disintegration theorem, we may write $\mu = \nu \otimes \mu_x$, where $\nu \in \mathcal{M}(D)$ and $\mu_x \in \mathcal{P}(\mathbb{R}^3)$ for $\nu$-almost all $x$. We obtain
\begin{equation}\label{eq:discreteintegrationbyparts4}
\int_{D} \bigg( \int_{\mathbb{R}^n} \psi(z) \, d\mu_x(z) \bigg) \, \varphi(x) \, d\nu(x) =
\end{equation}
\begin{equation*}
= \int_D \bigg(\sum_{i=1}^m \int_{\mathbb{R}^n} \chi_{U(0,1)}(z) \, \psi(z) \, z_i \, X_i f \,  d\mathcal{L}^n(z) \bigg) \, \varphi(x) \, d\mathcal{L}^n(x).
\end{equation*}
Hence, in the sense of measures as functionals on the space of continuous functions, we have that
\begin{equation}\label{eq:discreteintegrationbyparts5}
\bigg( \int_{\mathbb{R}^n} \psi(z) \, d\mu_x(z) \bigg) \nu = \sum_{i=1}^m \bigg( \int_{\mathbb{R}^n} \chi_{U(0,1)}(z) \, \psi(z) \, z_i \, d\mathcal{L}^n(z) \bigg) X_i f.
\end{equation}
Let us take $\widetilde{\psi} \in \mathcal{D}(\mathbb{R}^3)$ be a function such that $\widetilde{\psi} \equiv 1$ in $U(0,1)$. For $i = 1, ..., m$, we take $\psi(z) = \widetilde{\psi}(z) z_i$. Since the unit ball in the box distance in $\mathbb{G}$ is symmetric in the horizontal directions, for $i,j = 1,...,m$ such that $i \neq j$ we have
\begin{equation*}
\int_{\mathbb{R}^n} \chi_{U(0,1)} (z) \, z_i \, z_j \, \widetilde{\psi}(z) \, d\mathcal{L}^n(z) = \int_{\mathbb{R}^n} \chi_{U(0,1)} (z) \, z_i \, z_j \, d\mathcal{L}^n(z) = 0.
\end{equation*}
Then equation \eqref{eq:discreteintegrationbyparts5} takes form
\begin{equation}\label{eq:discreteintegrationbyparts6}
\bigg( \int_{\mathbb{R}^n} \widetilde{\psi}(z) \, z_i \, d\mu_x(z) \bigg) \nu = \bigg( \int_{\mathbb{R}^n} \chi_{U(0,1)}(z) \, z_i^2 \, d\mathcal{L}^n(z) \bigg) X_i f.
\end{equation}
The right hand side is $X_i f$ multiplied by a positive number; the left hand side is a bounded Radon measure. Hence the horizontal gradient of $f$ is a bounded Radon measure, so $f \in BV_X(D)$. Moreover, we see that
\begin{equation}\label{eq:structureofmu}
\mu(x,z) = \sum_{i=1}^m X_i f(x) \, \chi_{U(0,1)}(z) \, z_i \, \mathcal{L}^n(z),
\end{equation}
which proves \eqref{eq:structureofthederivative}.
\end{proof}

\subsection{Convergence of nonlocal least gradient functions}

Let us rewrite the definition of the solution to the nonlocal least gradient problem in the setting of Carnot groups. Denote by $\Omega_\varepsilon$ the set
$$ \Omega_\varepsilon = \bigg\{ x \in X: d_{cc}(x,\Omega) < \varepsilon \bigg\}.$$
Notice that for each $\varepsilon \in (0,1)$ we have $\Omega_\varepsilon \subset \Omega_1$; we will use $\Omega_1$ as a domain on which we will prove the necessary uniform estimates. Now, denote by $u_\psi$ the function
$$ u_\psi (x) = \twopartdef{u(x)}{x \in \Omega,}{\psi(x)}{x \in X \backslash \Omega.}$$
Now, we recall Definition \ref{Defi.1.var} and rewrite it in this setting:

\begin{definition}\label{dfn:nonlocalincarnotgroups} 
{\rm Let $\psi\in L^1(\Omega_\varepsilon \backslash \Omega)$. We say that $u_\varepsilon \in L^1(\Omega)$ is a solution to the nonlocal least gradient problem for the $\varepsilon$-step random walk, if there exists $g_\varepsilon \in L^\infty(\Omega_\varepsilon \times \Omega_\varepsilon)$ with
$\|g_\varepsilon \|_\infty\leq 1$ verifying
\begin{equation}\label{lasim.carnot}
g_\varepsilon(x,y)=-g_\varepsilon(y,x) \qquad \mbox{almost everywhere in } \Omega_\varepsilon \times \Omega_\varepsilon,
\end{equation}
\begin{equation}\label{1-lapla2.var.carnot}
g_\varepsilon(x,y)\in \mbox{sign}((u_\varepsilon)_\psi(y)-(u_\varepsilon)_\psi(x)) \qquad \mbox{almost everywhere in } \Omega_\varepsilon \times \Omega_\varepsilon,
\end{equation}
\begin{equation}\label{1-lapla.var.carnot}
-\int_{U(x,\varepsilon)} g_\varepsilon(x,y) \,d\mathcal{L}^n(y)= 0 \qquad \mbox{for almost every } x\in \Omega.
\end{equation}
}
\end{definition}

\begin{theorem}\label{thm:connectionwithlocal}
Suppose that $\Omega$ is an $X$-regular bounded domain in $\mathbb{G}$. Let $\psi \in BV_X(\mathbb{G}) \cap L^\infty(\mathbb{G})$. Let $u_\varepsilon$ be a sequence of solutions to the nonlocal least gradient problem corresponding to boundary data $\psi$ in the sense of Definition \ref{dfn:nonlocalincarnotgroups}. Then, on a subsequence, we have $u_\varepsilon \rightharpoonup u$ in $L^1(\Omega)$, where $u$ is a solution of the (local) least gradient problem on $\Omega$ with boundary data $h = T^- \psi$ in the sense of Definition \ref{dfn:1laplacecarnot}. Here, $T^- \psi$ denotes the one-sided trace of $\psi$ from $\mathbb{G} \backslash \overline{\Omega}$.
\end{theorem}

\begin{proof}
We represent the group $\mathbb{G}$ in exponential coordinates. Since we assumed that $\psi \in BV_X(\mathbb{G}) \cap L^\infty(\mathbb{G})$, using an estimate proved in \cite[Proposition 2.12]{GM2019} (see also \cite[Theorem 3.1]{MMS}) on a subsequence (still denoted by $u_\varepsilon$) we have
\begin{equation}\label{eq:uniformestimateinfinalproof}
\int_{\Omega_1} \dashint_{U(x,\varepsilon)} |(u_\varepsilon)_\psi(y) - (u_\varepsilon)_\psi(x)| \, d\mathcal{L}^n(x) \, d\mathcal{L}^n(y) \leq M \varepsilon.
\end{equation}
Since we assumed that $\psi \in L^\infty(\mathbb{G})$, it is easy to see that also $u_\varepsilon \in L^\infty(\mathbb{G})$ and $\| (u_\varepsilon)_\psi \|_\infty \leq \| \psi \|_\infty$ (for instance, see the proofs of \cite[Theorems 2.9, 2.10]{GM2019}). In particular, $(u_\varepsilon)_\psi \rightharpoonup u_\psi$ weakly in $L^1(\Omega_1)$ on a subsequence (still denoted by $\varepsilon$). By \eqref{eq:uniformestimateinfinalproof} the sequence $(u_\varepsilon)_\psi$ on the set $\Omega_1$ satisfies the assumptions of Theorem \ref{thm:amrtincarnotgroups}, so $u \in BV_X(\Omega)$ and
\begin{equation}\label{eq:structureofthederivativeinthefinalproof}
\chi_{U(0,1)}(z) \, \chi_D(x \delta_\varepsilon(z)) \, \frac{(u_\varepsilon)_\psi(x \delta_\varepsilon(z)) - (u_\varepsilon)_\psi(x)}{\varepsilon} \rightharpoonup \chi_{U(0,1)}(z) \, \cdot \, \langle z, Xu \rangle.
\end{equation}
Moreover, on a subsequence (still denoted by $\varepsilon$) we have that
\begin{equation}\label{eq:definitionoflambda}
\chi_{U(0,1)}(z) \, \chi_D(x \delta_\varepsilon(z)) \, g_\varepsilon(x, x \delta_\varepsilon(z)) \rightharpoonup \Lambda(x,z)
\end{equation}
weakly* in $L^\infty(\Omega_1 \times \mathbb{R}^n)$, as the sequence above is bounded from above by $1$. In particular, we have $\Lambda(x,z) \leq \chi_{U(0,1)}(z)$ almost everywhere in $\Omega_1$. 

Take $v \in C_c^\infty(\Omega)$. We multiply \eqref{1-lapla.var.carnot} by $v(x)$ and integrate over $\Omega$. We obtain
\begin{equation}
\int_\Omega \int_{U(x,\varepsilon)} g_\varepsilon(x,y) \, v(x) \,d\mathcal{L}^n(y) \, d\mathcal{L}^n(x) = 0.
\end{equation}
We use the antisymmetry of $g_\varepsilon$ to obtain
\begin{equation}
\int_\Omega \int_{U(x,\varepsilon)} \chi_\Omega(y) g_\varepsilon(x,y) \, (v(y) - v(x)) \,d\mathcal{L}^n(y) \, d\mathcal{L}^n(x) = 0.
\end{equation}
Now, we change variables. We set $z = \delta_{\varepsilon^{-1}}(x^{-1}y)$, so that $z \in U(0,1)$ and $y = x \delta_\varepsilon(z)$. In particular, $d\mathcal{L}^n(y) = \varepsilon^Q d\mathcal{L}^n(z)$. The equation above takes form
\begin{equation}
\int_\Omega \int_{U(0,1)} \chi_\Omega(x \delta_\varepsilon(z)) \, g_\varepsilon(x,x \delta_\varepsilon(z)) \, (v(x \delta_\varepsilon(z)) - v(x)) \, \varepsilon^Q \, d\mathcal{L}^n(z) \, d\mathcal{L}^n(x) = 0.
\end{equation}
We represent the integration over $U(0,1)$ using an indicator function and divide both sides of the equation by $\varepsilon^{Q+1}$. We obtain
\begin{equation}
\int_\Omega \int_{\mathbb{R}^n} \chi_{U(0,1)}(z) \,\chi_\Omega(x \delta_\varepsilon(z)) \, g_\varepsilon(x,x \delta_\varepsilon(z)) \, \frac{v(x \delta_\varepsilon(z)) - v(x)}{\varepsilon} \, d\mathcal{L}^n(z) \, d\mathcal{L}^n(x) = 0.
\end{equation}
We pass to the limit with $\varepsilon \rightarrow 0$. We use \eqref{eq:definitionoflambda} and Lemma \ref{lem:horizontalgradient}; the first one gives us weak convergence of the first three factors to $\Lambda(x,z)$, while the second one gives us strong convergence of the last factor to $\langle z, Xv \rangle$. We put these results together to see that
\begin{equation}\label{eq:almostzerodivergence}
\int_\Omega \int_{\mathbb{R}^n} \Lambda(x,z) \, \langle z, Xv \rangle \, d\mathcal{L}^n(z) \, d\mathcal{L}^n(x) = 0
\end{equation}
for all $v \in C_c^\infty(\Omega)$. Set $\zeta = (\zeta_1, ..., \zeta_m)$ to be the vector field defined by
\begin{equation}
\zeta_i(x) = \frac{1}{C_{\mathbb{G}}} \int_{\mathbb{R}^n} \Lambda(x,z) \, z_i \, d\mathcal{L}^n(z) \qquad \mbox{for } i=1,...,m,
\end{equation}
where $C_{\mathbb{G}}$ is a constant depending only on the Carnot group $\mathbb{G}$ and defined by the formula
\begin{equation}
C_{\mathbb{G}} := \int_{U(0,1)} |z_1| \, d\mathcal{L}^n(z) = ... = \int_{U(0,1)} |z_m| \, d\mathcal{L}^n(z),
\end{equation}
where the equalities follow from the fact that the unit ball in the box distance in $\mathbb{G}$ is symmetric in horizontal directions. The constant $C_{\mathbb{G}}$ is chosen in such a way so that $\| \zeta \|_{L^\infty(\Omega_1; \mathbb{R}^m)} \leq 1$: given a vector $\xi \in \mathbb{R}^n \backslash \{ 0 \}$ lying in the space spanned by the horizontal directions, i.e. $\xi \in \mbox{lin}(e_1,...,e_m)$, we denote by $R$ the rotation preserving the origin such that $\xi = |\xi| \cdot R e_1$. We set $z = Ry$, so that $d\mathcal{L}^n(z) = d\mathcal{L}^n(y)$, and calculate
\begin{equation}
\langle \zeta, \xi \rangle = \frac{1}{C_{\mathbb{G}}} \int_{\mathbb{R}^n} \Lambda(x,z) \,  z \cdot \xi \, d\mathcal{L}^n(z) = \frac{1}{C_{\mathbb{G}}} \int_{\mathbb{R}^n} \Lambda(x,Ry) \, Ry \cdot \xi \, d\mathcal{L}^n(y) =
\end{equation}
\begin{equation}
= \frac{1}{C_{\mathbb{G}}} \int_{\mathbb{R}^n} \Lambda(x,Ry) \, y \cdot R^{-1} \xi \, d\mathcal{L}^n(y) = \frac{1}{C_{\mathbb{G}}} \int_{\mathbb{R}^n} \Lambda(x,Ry) \, y_1 \, |\xi| \, d\mathcal{L}^n(y),
\end{equation}
so
\begin{equation}
|\langle \zeta, \xi \rangle| \leq \frac{1}{C_{\mathbb{G}}} \int_{\mathbb{R}^n} \Lambda(x,Ry) \, |y_1| \, |\xi| \, d\mathcal{L}^n(y) = |\xi|.
\end{equation}
Therefore, $\| \zeta \|_{L^\infty(\Omega_1; \mathbb{R}^m)} \leq 1$.

Coming back to equation \eqref{eq:almostzerodivergence}, it reduces to
\begin{equation}
\int_\Omega \zeta(x) \cdot Xv(x) \, d\mathcal{L}^n(x) = 0,
\end{equation}
so $\mbox{div}_X(\zeta) = 0$ as a distribution in $\Omega$. In particular, by Riesz representation theorem $\mbox{div}_X(\zeta)$ is a (zero) Radon measure; as it is absolutely continuous with respect to $\mathcal{L}^n$ and we can say that $\mbox{div}_X(\zeta) \in L^\infty(\Omega)$ with $\mbox{div}_X(\zeta) = 0$. 

It remains to show that
\begin{equation}
(\zeta, Xu) = |Xu| \qquad \mbox{as measures in }\Omega;
\end{equation}
\begin{equation}
[\zeta \cdot \nu]_X \in \mbox{sign}(h-u) \qquad |\partial\Omega|_X-\mbox{a.e. on }\partial\Omega.
\end{equation}
Let us choose a function $w \in C^\infty(\Omega) \cap W^{1,1}_X(\Omega)$ such that $Tw = h$; this is possible due to Theorem \ref{thm:vittoneextension}. Now, set $v_{\varepsilon} = (u_\varepsilon)_\psi - w_\psi$. We take the property \eqref{1-lapla.var.carnot} of the solution, multiply it by $v_{\varepsilon}$ and integrate over $\Omega_1$ to obtain
\begin{equation}
\int_{\Omega_1} \int_{U(x,\varepsilon)} g_\varepsilon(x,y) \, v_\varepsilon(x) \,d\mathcal{L}^n(y) \, d\mathcal{L}^n(x) = 0.
\end{equation}
We proceed similarly as we did in the proof of equation \eqref{eq:almostzerodivergence}: we use the antisymmetry of $g_\varepsilon$, change variables to $z = \delta_{\varepsilon^{-1}}(x^{-1}y)$, represent the integration over $U(0,1)$ using an indicator function and divide both sides of the equation by $\varepsilon^{Q+1}$. We obtain
\begin{equation}
\int_{\Omega_1} \int_{\mathbb{R}^n} \chi_{U(0,1)}(z) \,\chi_\Omega(x \delta_\varepsilon(z)) \, g_\varepsilon(x,x \delta_\varepsilon(z)) \, \frac{v_\varepsilon(x \delta_\varepsilon(z)) - v_\varepsilon(x)}{\varepsilon} \, d\mathcal{L}^n(z) \, d\mathcal{L}^n(x) = 0.
\end{equation}
We divide the above equation into two parts: we set
\begin{equation}
H^1_\varepsilon = \int_{\Omega_1} \int_{\mathbb{R}^n} \chi_{U(0,1)}(z) \,\chi_\Omega(x \delta_\varepsilon(z)) \, g_\varepsilon(x,x \delta_\varepsilon(z)) \, \frac{(u_\varepsilon)_\psi(x \delta_\varepsilon(z)) - (u_\varepsilon)_\psi(x)}{\varepsilon} \, d\mathcal{L}^n(z) \, d\mathcal{L}^n(x) =
\end{equation}
\begin{equation}
= \int_{\Omega_1} \int_{\mathbb{R}^n} \chi_{U(0,1)}(z) \,\chi_\Omega(x \delta_\varepsilon(z)) \, \bigg| \frac{(u_\varepsilon)_\psi(x \delta_\varepsilon(z)) - (u_\varepsilon)_\psi(x)}{\varepsilon} \bigg| \, d\mathcal{L}^n(z) \, d\mathcal{L}^n(x),
\end{equation}
where equality follows from property \eqref{1-lapla2.var.carnot}, and
\begin{equation}
H^2_\varepsilon = - \int_{\Omega_1} \int_{\mathbb{R}^n} \chi_{U(0,1)}(z) \, \chi_\Omega(x \delta_\varepsilon(z)) \, g_\varepsilon(x,x \delta_\varepsilon(z)) \, \frac{w_\psi(x \delta_\varepsilon(z)) - w_\psi(x)}{\varepsilon} \, d\mathcal{L}^n(z) \, d\mathcal{L}^n(x).
\end{equation}
By definition, we have $H^1_\varepsilon + H^2_\varepsilon = 0$. We will estimate the limits of both expressions separately. By \eqref{eq:structureofthederivativeinthefinalproof} and the definition of $C_{\mathbb{G}}$, we have
\begin{equation}
\liminf_{\varepsilon \rightarrow 0} H^1_\varepsilon \geq C_{\mathbb{G}} \int_{\Omega_1} |Xu_\psi| = C_{\mathbb{G}} \int_{\Omega} |Xu| + C_{\mathbb{G}} \int_{\partial \Omega} |Tu - h| \, d|\partial\Omega|_X + C_{\mathbb{G}} \int_{\Omega_1 \backslash \overline{\Omega}} |X \psi|,
\end{equation}
where the equality follows from Theorem \ref{thm:vittonesplittingthegradient}. Similarly, we will estimate the limit of $H^2_\varepsilon$. Since $w \in C^\infty(\Omega) \cap W^{1,1}_X(\Omega)$, \eqref{eq:almostzerodivergence} implies
\begin{equation}
\lim_{\varepsilon \rightarrow 0} H^2_\varepsilon = - C_{\mathbb{G}} \int_{\Omega_1} \int_{\mathbb{R}^n} \Lambda(x,z) \, \langle z, Xw_\psi \rangle \, d\mathcal{L}^n(z) \, d\mathcal{L}^n(x) = 
\end{equation}
\begin{equation}
= - C_{\mathbb{G}} \int_{\Omega_1} \zeta(x) \cdot Xw_\psi \, d\mathcal{L}^n(x).
\end{equation}
We pass to the limit with $\varepsilon \rightarrow 0$ in the equation $H^1_\varepsilon + H^2_\varepsilon = 0$, divide by $C_{\mathbb{G}}$ and obtain
\begin{equation}\label{eq:almostfinalestimate}
0 \geq \int_{\Omega} |Xu| + \int_{\partial \Omega} |Tu - h| \, d|\partial\Omega|_X + \int_{\Omega_1 \backslash \overline{\Omega}} |X \psi| - \int_{\Omega_1} \zeta(x) \cdot Xw_\psi \, d\mathcal{L}^n(x).
\end{equation}
We use the Green's formula (Theorem \ref{thm:greenformulacarnot}) and the fact that $\mbox{div}_X(\zeta) = 0$ as a function in $L^\infty(\Omega)$ to take a closer look at the last summand:
\begin{equation}
- \int_{\Omega_1} \zeta(x) \cdot Xw_\psi \, d\mathcal{L}^n(x) = - \int_{\Omega} \zeta(x) \cdot Xw \, d\mathcal{L}^n(x) - \int_{\Omega_1 \backslash \overline{\Omega}} \zeta(x) \cdot X\psi \, d\mathcal{L}^n(x) = 
\end{equation}
\begin{equation}
= - \int_{\partial\Omega} [\zeta, \nu]_X \, h \, d|\partial\Omega|_X - \int_{\Omega_1 \backslash \overline{\Omega}} \zeta(x) \cdot X\psi \, d\mathcal{L}^n(x).
\end{equation}
Since $\| \zeta \|_{L^\infty(\Omega_1; \mathbb{R}^m)} \leq 1$, we have that
\begin{equation}
\int_{\Omega_1 \backslash \overline{\Omega}} |X\psi| \, d\mathcal{L}^n(x) \geq \bigg| \int_{\Omega_1 \backslash \overline{\Omega}} \zeta(x) \cdot X\psi \, d\mathcal{L}^n(x) \bigg|,
\end{equation}
so the estimate \eqref{eq:almostfinalestimate} reduces to
\begin{equation}
0 \geq \int_{\Omega} |Xu| + \int_{\partial \Omega} |Tu - h| \, d|\partial\Omega|_X - \int_{\partial\Omega} [\zeta, \nu]_X \, h \, d|\partial\Omega|_X.
\end{equation}
Again, we use Green's formula (Theorem \ref{thm:greenformulacarnot}) and the fact that $\mbox{div}_X(\zeta) = 0$ as a function in $L^\infty(\Omega)$ to obtain
\begin{equation}
0 \geq \int_{\Omega} |Xu| + \int_{\partial \Omega} |Tu - h| \, d|\partial\Omega|_X - \int_\Omega (\zeta, Xu) + \int_{\partial\Omega} [\zeta, \nu]_X \, Tu \, d|\partial\Omega|_X - \int_{\partial\Omega} [\zeta, \nu]_X \, h \, d|\partial\Omega|_X.
\end{equation}
Since $\| \zeta \|_{L^\infty(\Omega_1; \mathbb{R}^m)} \leq 1$, by Proposition \ref{prop:boundonanzelottipairing} we have $|(\zeta,Xu)| \leq |Xu|.$ Finally, since we assumed $\Omega$ to be $X$-regular, Theorem \ref{thm:definitionofweaktrace} implies that $|[\zeta, \nu]_X| \leq 1$. We rewrite the equation above and obtain
\begin{equation}
\int_{\partial \Omega} |Tu - h| \, d|\partial\Omega|_X \leq - \int_{\Omega} |Xu| + \int_\Omega (\zeta, Xu) + \int_{\partial\Omega} [\zeta, \nu]_X \, (h-Tu) \, d|\partial\Omega|_X \leq
\end{equation}
\begin{equation}
\leq \int_{\partial\Omega} [\zeta, \nu]_X \, (h-Tu) \, d|\partial\Omega|_X \leq \int_{\partial \Omega} |Tu - h| \, d|\partial\Omega|_X,
\end{equation}
so all inequalities above are equalities. In particular, $(\zeta, Xu) = |Xu|$ as measures and $[\zeta, \nu]_X \in \mbox{sign}(h - u)$ $|\partial\Omega|_X$-a.e. on $\partial\Omega$, so $u$ is a solution to the (local) least gradient problem with boundary data $h$ in the sense of Definition \ref{dfn:1laplacecarnot}.
\end{proof}

{\bf Acknowledgements.} This work has been partially supported by the research
project no. 2017/27/N/ST1/02418 funded by the National Science Centre, Poland. The motivation for writing this paper originated during my visit to the Universitat de Val\`encia; I wish to thank them for their hospitality and Jos\'e M. Maz\'on for his support.

\bibliographystyle{siam}%
\bibliography{WG-references}%

\begin{thebibliography}{10}

\bibitem{AMRT1}
{\sc F.~Andreu, J.~Maz\'{o}n, J.~Rossi, and J.~Toledo}, {\em A nonlocal
  p-laplacian evolution equation with a nonhomogeneus {D}irichlet boundary
  conditions}, SIAM J. Math. Anal., 40 (2009), pp.~1815--1851.

\bibitem{AMRT}
{\sc F.~Andreu-Vaillo, J.~Maz\'{o}n, J.~Rossi, and J.~Toledo}, {\em Nonlocal
  Diffusion Problems}, Mathematical Surveys and Monographs, vol. 165, AMS,
  2010.

\bibitem{Anz}
{\sc G.~Anzelotti}, {\em Pairings between measures and bounded functions and
  compensated compactness}, Ann. di Matematica Pura ed Appl. IV, 135 (1983),
  pp.~293--318.

\bibitem{BGG}
{\sc E.~Bombieri, E.~de~Giorgi, and E.~Giusti}, {\em Minimal cones and the
  {Bernstein} problem}, Invent. Math., 7 (1969), pp.~243--268.

\bibitem{BBM}
{\sc J.~Bourgain, H.~Brezis, and P.~Mironescu}, {\em Another look at {S}obolev
  spaces}, in Optimal Control and Partial Diferential Equations, J.~L.~M.
  et~al., ed., Amsterdam, 2001, IOS Press, pp.~439--455.

\bibitem{DG2019}
{\sc S.~Dweik and W.~G\'{o}rny}, {\em Least gradient problem on annuli},
  arXiv:1908.09113,  (2019).

\bibitem{DS}
{\sc S.~Dweik and F.~Santambrogio}, {\em {$L^p$} bounds for
  boundary-to-boundary transport densities, and {$W^{1,p}$} bounds for the {BV}
  least gradient problem in {2D}}, Calc. Var. Partial Differential Equations,
  58 (2019), p.~31.

\bibitem{GM2019}
{\sc W.~G\'{o}rny and J.~Maz\'{o}n}, {\em Least gradient functions on metric
  random walk spaces}, arXiv:1912.12731,  (2019).

\bibitem{GRS2017NA}
{\sc W.~G\'{o}rny, P.~Rybka, and A.~Sabra}, {\em Special cases of the planar
  least gradient problem}, Nonlinear Anal., 151 (2017), pp.~66--95.

\bibitem{HZ}
{\sc P.~Hajlasz and S.~Zimmerman}, {\em Geodesics in the {H}eisenberg group},
  Anal. Geom. Metr. Spaces, 3 (2015), pp.~325--337.

\bibitem{JMN}
{\sc R.~Jerrard, A.~Moradifam, and A.~Nachman}, {\em Existence and uniqueness
  of minimizers of general least gradient problems}, J. Reine Angew. Math., 734
  (2018), pp.~71--97.

\bibitem{Juu}
{\sc P.~Juutinen}, {\em p-harmonic approximation of functions of least
  gradient}, Indiana Univ. Math. J., 54 (2005), pp.~1015--1029.

\bibitem{LeD}
{\sc E.~Le~Donne}, {\em A primer on {C}arnot groups: homogenous groups,
  {C}arnot-{C}arathéodory spaces, and regularity of their isometries}, Anal.
  Geom. Metr. Spaces, 5 (2017), pp.~116--137.

\bibitem{MMS}
{\sc N.~Marola, M.~Miranda~Jr., and N.~Shanmugalingam}, {\em Characterizations
  of sets of finite perimeter using heat kernels in metric spaces}, Potential
  Anal., 45 (2016), pp.~609--633.

\bibitem{MPRT}
{\sc J.~Maz\'on, M.~Perez-Llanos, J.~Rossi, and J.~Toledo}, {\em A nonlocal
  $1$-{L}aplacian problem and median values}, Publ. Mat., 60 (2016),
  pp.~27--53.

\bibitem{MRL}
{\sc J.~Maz\'on, J.~Rossi, and S.~Segura~de Le\'on}, {\em Functions of least
  gradient and 1-harmonic functions}, Indiana Univ. Math. J., 63 (2014),
  pp.~1067--1084.

\bibitem{MRT}
{\sc J.~Maz\'on, J.~Rossi, and J.~Toledo}, {\em Nonlocal Perimeter, Curvature
  and Minimal Surfaces for Measurable Sets}, Frontiers in Mathematics,
  Birkh\"auser, Basel, DOI: 10.1007/978-3-030-06243-9, 2019.

\bibitem{MST1}
{\sc J.~Maz\'on, M.~Solera, and J.~Toledo}, {\em The total variation flow in
  metric random walk spaces}, Calc. Var. Partial Differential Equations, to
  appear,  (2019).

\bibitem{Mon00}
{\sc R.~Monti}, {\em Some properties of {C}arnot-{C}arath\'eodory balls in the
  {H}eisenberg group}, Rend. Lincei Mat. Appl., 11 (2000), pp.~155--167.

\bibitem{MM}
{\sc R.~Monti and D.~Morbidelli}, {\em Trace theorems for vector fields}, Math.
  Z.,  (2002), pp.~747--776.

\bibitem{SWZ}
{\sc P.~Sternberg, G.~Williams, and W.~Ziemer}, {\em Existence, uniqueness, and
  regularity for functions of least gradient}, J. Reine Angew. Math., 430
  (1992), pp.~35--60.

\bibitem{VittPhD}
{\sc D.~Vittone}, {\em Submanifolds in {C}arnot groups}, PhD thesis, Theses of
  Scuola Normale Superiore di Pisa (New Series), Edizioni della Normale, Pisa,
  2010.

\bibitem{Vitt}
\leavevmode\vrule height 2pt depth -1.6pt width 23pt, {\em Lipschitz surfaces,
  perimeter and trace theorems for {BV} functions in {C}arnot-{C}arath\'eodory
  spaces}, Ann. Sc. Norm. Super. Pisa Cl. Sci., IX (2012), pp.~939--998.

\end{thebibliography}

\end{document}